\documentclass[a4paper,english]{jnsao}
\usepackage[utf8]{inputenc}
\usepackage[T1]{fontenc}
\usepackage{csquotes}
\usepackage{babel}
\usepackage{tikz}
\usepackage{float}
\usepackage{enumitem}
\usepackage{comment}
\usepackage{asymptote}
\usepackage[nameinlink,capitalize]{cleveref}
\newcommand{\term}{\emph}

\newcommand{\field}[1]{\mathbb{#1}}
\newcommand{\N}{\mathbb{N}}

\newcommand{\R}{\field{R}}
\newcommand{\extR}{\overline \R}
\newcommand{\B}{B}

\newcommand{\norm}[1]{\|#1\|}

\newcommand{\abs}[1]{|#1|}

\newcommand{\inv}[1]{#1^{-1}}
\newcommand{\grad}{\nabla}
\newcommand{\freevar}{\,\boldsymbol\cdot\,}

\newcommand{\Union}\bigcup
\newcommand{\Isect}\bigcap
\newcommand{\union}\cup
\newcommand{\isect}\cap
\newcommand{\bigunion}\bigcup
\newcommand{\bigisect}\bigcap

\newcommand{\defeq}{:=}

\DeclareRobustCommand{\upto}{{{\mathchoice%
            {\rotatebox[origin=c]{20}{$\to$}}
            {\rotatebox[origin=c]{20}{$\to$}}
            {\rotatebox[origin=c]{20}{\scalebox{0.75}{$\to$}}}
            {\rotatebox[origin=c]{20}{\scalebox{0.6}{$\to$}}}
}}}
\newcommand{\subdiff}{\partial}

\DeclareMathOperator*{\argmin}{arg\,min}

\DeclareMathOperator{\closure}{cl}

\DeclareMathOperator{\Dom}{dom}

\DeclareMathOperator{\diag}{diag}
\DeclareMathOperator{\graph}{graph}

\makeatletter
\def \uminus@sym{\setbox0=\hbox{$\cup$}\rlap{\hbox 
        to\wd0{\hss\raise0.5ex\hbox{$\scriptscriptstyle{-}$}\hss}}\box0}
    \def \uminus    {\mathrel{\uminus@sym}}
\makeatother

\newcommand{\iprod}[2]{\langle #1,#2\rangle}

\makeatletter
\def \weaktostar@sym{\setbox0=\hbox{$\rightharpoonup$}\rlap{\hbox 
        to\wd0{\hss\raise1ex\hbox{$\scriptscriptstyle{*\,}$}\hss}}\box0}
    \def \weaktostar    {\mathrel{\weaktostar@sym}}
\makeatother

\def\linear{\mathbb{L}}

\newcommand{\setto}{\rightrightarrows}

\def\extR{\overline \R}

\def\realopt#1{\widehat #1}
\def\this#1{#1^i}
\def\nexxt#1{#1^{i+1}}
\def\overnext#1{\bar #1^{i+1}}

\def\realoptu{{\realopt{u}}}

\def\realoptx{{\realopt{x}}}
\def\realoptq{{\realopt{q}}}

\def\realopty{{\realopt{y}}}

\def\nextu{\nexxt{u}}
\def\nextx{\nexxt{x}}

\def\nexty{\nexxt{y}}

\def\thisu{\this{u}}
\def\thisx{\this{x}}

\def\thisy{\this{y}}

\def\overnextx{\overnext{x}}

\def\tauTest{\phi}

\def\sigmaTest{\psi}

\def\gap{\mathcal{G}}

\def\Space{U}
\def\SpaceTwo{W}

\newcommand{\Test}{Z}
\newcommand{\Precond}{M}

\DeclareFontFamily{U}{mathx}{\hyphenchar\font45}
\DeclareFontShape{U}{mathx}{m}{n}{<-> mathx10}{}
\DeclareSymbolFont{mathx}{U}{mathx}{m}{n}
\DeclareMathAccent{\widebar}{0}{mathx}{"73}

\def\realoptw{\realopt{w}}

\DeclareMathOperator{\dist}{dist}

\def\bar{\widebar}
\def\tilde{\widetilde}

\renewcommand{\B}{\mathbb{B}}

\manuscriptcopyright{}
\manuscriptlicense{}
\shorttitle{Partial subregularity}
\shortauthor{T.~Valkonen}

\def\thetitle{Preconditioned proximal point methods and notions of partial subregularity}
\title{\thetitle}

\AtBeginDocument{
    \hypersetup{
        pdftitle = {\thetitle},
        pdfauthor = {T.~Valkonen}
    }
}

\begin{document}

\date{2017-11-15 (revised 2020-02-28)}
\author{
    Tuomo Valkonen\thanks{%
        Department of Mathematics and Statistics, University of Helsinki, Finland \emph{and} ModeMat, Escuela Politécnica Nacional, Quito, Ecuador; Most of this work was completed while at the Department of Mathematical Sciences, University of Liverpool, United Kingdom. \email{tuomo.valkonen@iki.fi} \orcid{0000-0001-6683-3572}
    }
}

\maketitle

\begin{abstract}
	Based on the needs of convergence proofs of preconditioned proximal point methods, we introduce notions of partial strong submonotonicity and partial (metric) subregularity of set-valued maps. We study relationships between these two concepts, neither of which is generally weaker or stronger than the other one.
	For our algorithmic purposes, the novel submonotonicity turns out to be easier to employ than more conventional error bounds obtained from subregularity.
	Using strong submonotonicity, we demonstrate the linear convergence of the Primal-Dual Proximal splitting method to some strictly complementary solutions of example problems from image processing and data science. This is without the conventional assumption that all the objective functions of the involved saddle point problem are strongly convex.
\end{abstract}

\newenvironment{demonstration}[1][Verification]{\begin{proof}[#1]}{\end{proof}}

\section{Introduction}
\label{sec:intro}

What is the weakest useful form of regularity of a set-valued map $H$? In particular, if $0 \in H(\realoptu)$ for $\realoptu=(\realoptx,\realopty)$ encodes optimality conditions of a saddle point problem
\begin{equation}
	\label{eq:saddle-intro}
    \min_{x \in X} \max_{y \in Y}~G(x) + \iprod{Kx}{y} - F^*(y),
\end{equation}
in Hilbert spaces $X$ and $Y$, what regularity property is useful for showing the fast convergence of optimisation methods? In this work, we study this question in relation to preconditioned proximal point methods \cite{tuomov-proxtest}.

A starting point for the regularity of set-valued maps is to extend the Lipschitz property of single-valued maps. One such approach is the \term{Aubin}, \term{pseudo-Lipschitz}, or \term{Lipschitz-like} property of \cite{aubin1984lipschitz}. 
When we are interested in the stability of the optimality condition $0 \in H(\realoptu)$, it is more beneficial to study the Aubin property of the inverse $\inv H$. This is called the \term{metric regularity} of $H$. In this property, \emph{at $\realoptu$ for $\realoptw \in H(\realoptu)$}, both $u$ and $w$ are allowed to vary in the criterion
\[
	\kappa \dist(w, H(u)) \ge \dist(u, \inv H(w))
	\quad (u \in \mathcal{U},\, w \in \mathcal{W}),
\]
which is assumed to hold for some $\kappa > 0$, and neighbourhoods $\mathcal{U} \ni \realoptu$ and $\mathcal{W} \ni \realoptw$. We make standard notation such as “$\dist$” explicit in \cref{sec:notation}.
Metric regularity is equivalent to \term{openness at a linear rate} at $(\realoptu, \realoptw)$, and holds for smooth maps by the classical Lyusternik--Graves theorem; see, e.g., \cite{ioffe2017variational}.
It is too strong a property to be satisfied in many applications.
\term{Metric subregularity} \cite{ioffe1979regular,dontchev2004regularity} allows much more leeway for $H$ by fixing $w=\realoptw$. In other words, we require
\begin{equation}
	\label{eq:basic-subreg}
	\kappa \dist(\realoptw, H(u)) \ge \dist(u, \inv H(\realoptw))
	\quad (u \in \mathcal{U}).
\end{equation}
The counterpart of metric subregularity that relaxes the Aubin property is known as \term{calmness} or the \term{upper Lipschitz} property \cite{robinson1981continuity}.
We refer to the books \cite{aubin1990sva,mordukhovich2006variational,rockafellar-wets-va,dontchev2014implicit,ioffe2017variational} for further information on these and other related properties. These include the Mordukhovich criterion that allows verifying the Aubin property or metric regularity through coderivative considerations.


Monotonicity is an important basic property satisfied by the convex subdifferential.
For general set-valued maps $H: \Space \setto \Space$ on a Hilbert space $\Space$, it states
\[
    \iprod{q-q^*}{u-u^*} \ge 0
    \quad ( (u, q), (u^*, q^*) \in \graph H).
\]
Basic monotonicity does not help to achieve fast convergence of optimisation methods. This can be achieved with strong monotonicity
\[
    \iprod{q-q^*}{u-u^*} \ge \gamma\norm{u-u^*}^2
    \quad ( (u, q), (u^*, q^*) \in \graph H).
\]

In this work, in the spirit of subregularity, we want to relax strong monotonicity to \term{strong submonotonicity} at a point. To study the convergence of optimisation methods on subspaces, we will formulate both submonotonicity and subregularity in terms of operators: we introduce \term{partial (metric) subregularity} and \term{partial strong submonotonicity}.
In this theme, our \textbf{main contributions} are:
\begin{itemize}[label={--},nosep]
    \item In \cref{sec:submonotonicity} we introduce the new concept of \term{(partial, strong) submonotonicity}. Expanding from our previous work in \cite{tuomov-proxtest}, we apply it to the convergence proofs of abstract preconditioned proximal point methods.
    \item In \cref{sec:rates-from-error-bounds} we introduce partial variants of familiar error bounds and metric subregularity. Based on them we again formulate convergence results for abstract preconditioned proximal point methods.
    \item In \cref{sec:relationships} we study the relationship between submonotonicity and subregularity.
    \item In \cref{sec:saddle} we apply the work of the previous sections to obtain improved convergence results for the primal-dual proximal splitting (PDPS) of Chambolle and Pock \cite{chambolle2010first}.
\end{itemize}
Before this, in \cref{sec:notation}, we introduce notation.

There are further properties in the literature, related to our work.
\term{Strong metric subregularity} is a strengthening of metric subregularity. Its many properties are studied in \cite{cibulka2018strong}, along with introducing $q$-exponent versions. Particularly worth noting is that strong metric subregularity is invariant with respect to perturbations by smooth functions, while metric subregularity is not.
Minding our focus on monotonicity, of particular relevance is that for convex functions $f$, strong metric subregularity of the subdifferential amounts to local strong convexity.
Indeed, according to \cite{aragon2008characterization,artacho2013metric}, the (convex) subdifferential $\subdiff f$ is \emph{strongly} metrically subregular at $\realoptu$ for $\realoptw \in \subdiff f(\realoptu)$ if and only if there exists $\gamma>0$ and a neighbourhood $\mathcal{U} \ni \realoptu$ such that
\[
	f(u) - f(\realoptu) \ge \iprod{\realoptw}{u-\realoptu} + \gamma\norm{u-\realoptu}^2 \quad (u \in \mathcal{U}).
\]
Metric subregularity, on the other hand, amounts of measuring growth away from the entire set $\inv{[\subdiff f]}(\realoptw)$ instead of the single point $\realoptu$. This gives the corresponding criterion
\[
	f(u) - f(\realoptu) \ge \iprod{\realoptw}{u-\realoptu} + \gamma\dist^2(u, \inv{[\subdiff f]}(\realoptw)) \quad (u \in \mathcal{U}).
\]
In \cref{sec:relationships} we derive similar correspondences between partial metric subregularity and partial strong submonotonicity.

Some weaker and partial concepts of regularity have also been considered in the literature.
Of particular note is the directional metric subregularity of \cite{gfrerer2013directional}.
The idea here is to study necessary optimality conditions by only requiring metric regularity or subregularity to critical directions instead of all directions. A somewhat similar idea is considered in \cite{vanngai2015directional}.
We, by contrast, are interested in regularity on entire subspaces where our objective functions behave better than globally.
A form of directional metric subregularity more closely resembling our partial subregularity has been studied in \cite{durea2017directional,vanngai2012slopes}. This notion models directionality with subsets, while ours uses operators. Moreover, through the use of three rather than two operators, our definitions in \cref{sec:submonotonicity} split the partiality or directionality on two sides of the defining inequality. As we show in \cref{sec:relationships}, this yields a somewhat weaker property, sufficient for our needs.

Related to our objectives, partial calmness of bi-level programs with respect to perturbations of only certain optimality criteria have been studied in \cite{ye1995optimality}, see also \cite{henrion2011calmness}.
We have also ourselves studied partial strong convexity in \cite{tuomov-cpaccel,tuomov-blockcp} for the acceleration of optimisation methods on subspaces of strong convexity.
Moving from subspaces to manifolds, \cite{lewis2013partial,liang2014local} apply smoothness restricted to submanifolds to prove local linear convergence of optimisation algorithms to such submanifolds.

More generally, local linear convergence can be derived from error bounds, first introduced in \cite{luo1992error} for matrix splitting problems, and studied for other methods including the ADMM and the proximal point method, among others in \cite{han2013local,aspelmeier2016local,leventhal2009metric,li2012holder}. An alternative approach to the proximal point method is taken in \cite{aragon2012lyusternik} based on Lyusternik--Graves style estimates, while \cite{adly2015newton} presents an approach based on metric regularity to Newton's method for variational inclusions.
In \cite{zhou2017unified} a unified approach to error bounds for generic smooth constrained problems is introduced.
In \cref{sec:rates-from-error-bounds} we will do something roughly similar, but for nonsmooth optimisation.
Robinson stability \cite{robinson1976stability,gfrerer2017robinson} can also be seen as a stronger parametric version of error bounds, useful in the study of sensitivity of solution mappings. In this work, we are primarily interested in the application of our regularity criteria to the convergence of optimisation methods.

\section{Notation}
\label{sec:notation}

Let $X$ and $Y$ be Hilbert spaces. We denote by $\iprod{\freevar}{\freevar}$ and $\norm{\freevar}$ the corresponding inner products and induced norms.
We use $\linear(X; Y)$ to denote the space of bounded linear operators between $X$ and $Y$. We denote the identity operator by $I$.
For $T,S \in \linear(X; X)$, we write $T \ge S$ when $T-S$ is positive semidefinite.
Also for possibly non-self-adjoint $T \in \linear(X; X)$, we introduce the inner product and norm-like notations
\begin{equation}
    \label{eq:iprod-def}
    \iprod{x}{z}_T \defeq \iprod{Tx}{z},
    \quad
    \text{and}
    \quad
    \norm{x}_T \defeq \sqrt{\iprod{x}{x}_T}.
\end{equation}
For a set $A \subset X$, we write $\delta_A$ for the $\{0,\infty\}$-valued indicator function and define the distance
\[
    \dist_T(z, A) \defeq \inf_{u \in A } \norm{z-u}_T.
\]
We also write $\dist_T^2(z, A) \defeq \dist_T(z, A)^2$.

We denote the $\{0,\infty\}$-valued indicator function of a set $A \subset X$ by $\delta_A$, and write $A \ge 0$ if every element $t \in A$ satisfies $t \ge 0$. Also, we write $\iprod{A}{x}=\{\iprod{y}{x} \mid y \in A\}$ whenever $x \in X$, so that the inequality $\iprod{A}{x} \ge 0$ means $\iprod{y}{x} \ge 0$ for all $y \in A$.

Finally, we write $\B(x, \alpha)$ for the open ball of radius $\alpha$ around $x \in X$, and, for a possibly set-valued $H: X \setto Y$, the graph by $\graph H \defeq \{(x, y) \mid y \in H(x),\, x \in X\}$.

Let $f: X \setto \extR$. We write $\subdiff f$ for the convex subdifferential if $f$ convex. If $f$ is locally Lipschitz, we write $\subdiff_C$ for the Clarke subdifferential. 
We write $f^*$ for the Fenchel conjugate.

\section{Submonotonicity and preconditioned proximal point methods}
\label{sec:submonotonicity}

Our objective is to solve the variational inclusion $0 \in H(\realoptu)$, where $H: \Space \setto \Space$ is a set-valued map on a Hilbert space $\Space$.
Our strategy is for some iteration-dependent $\Precond_{i+1} \in \linear(\Space; \Space)$ the preconditioned proximal point method: given an initial iterate $u^0$, iteratively solve $\nextu$ from
\begin{equation}
    \label{eq:pp}
    \tag{PP}
    0 \in H(\nextu) + \Precond_{i+1}(\nextu-\thisu).
\end{equation}
In \cite{tuomov-proxtest} we developed general convergence theory of such methods and more general methods that include \emph{non-linear} preconditioners for the modelling forward steps. Here, for simplicity, we only consider linear preconditioning. 
By \cite{he2012convergence}, the primal-dual proximal splitting of \cite{chambolle2010first} is of the form \eqref{eq:pp}. We will discuss it in \cref{sec:saddle}. By \cite{tuomov-proxtest,clasonvalkonen2020nonsmooth} many more algorithms are of the form \eqref{eq:pp}.

We improve the results of \cite{tuomov-proxtest} to require weaker regularity of $H$ by only asking for convergence to the \emph{set} of roots $\inv H(0)$ instead of a specific (unique) root.
In this section, we specifically use submonotonicity that we define in \cref{sec:rates-from-monotonicity} while in \cref{sec:rates-from-error-bounds} we look at subregularity and error bounds. Before this, in \cref{sec:convest}, we provide the fundamental estimate for such convergence results.

\subsection{A convergence estimate}
\label{sec:convest}

The next lemma is the basis of our convergence results. The \term{testing operators} $\Test_{i+1}$ combined with the preconditioners $M_{i+1}$ form local metrics that measure the rate of convergence.

\begin{lemma}
    \label{lemma:main-convergence}
    Let $H: \Space \setto \Space$ on a Hilbert space $\Space$ with $\inv H(0)$ non-empty.
    Also let $\Precond_{i+1}, \Test_{i+1} \in \linear(\Space; \Space)$ with $\Test_{i+1}\Precond_{i+1}$ self-adjoint for all $i \in \N$.
    Suppose \eqref{eq:pp} is solvable for the iterates $\{\nextu\}_{i \in \N}$ given $u^0 \in U$. If
    \begin{multline}
        \label{eq:convergence-fundamental-condition-iter-h}
        \frac{1}{2}\norm{\nextu-\thisu}_{\Test_{i+1} \Precond_{i+1}}^2
        + \inf_{u^* \in \inv H(0)}\left(
            \frac{1}{2}\norm{\nextu-u^*}_{\Test_{i+1}\Precond_{i+1}}^2
            + \iprod{H(\nextu)}{\nextu-u^*}_{\Test_{i+1}}
        \right)
        \\
        \ge
        \frac{1}{2}\dist^2_{\Test_{i+2}\Precond_{i+2}}(\nextu; \inv H(0))
    \end{multline}
    for all $i \in \N$, then
    \begin{equation}
        \label{eq:convergence-result-main-h}
        \frac{1}{2}\dist_{\Test_{n+1}\Precond_{n+1}}^2(u^N, \inv H(0))
        \le
        \frac{1}{2}\dist_{\Test_{1}\Precond_{1}}^2(u^0, \inv H(0))
        \quad
        (n \ge 1).
    \end{equation}
\end{lemma}

\begin{proof}
    Let $u^* \in \inv H(0)$ be arbitrary.
    Inserting \eqref{eq:pp} into \eqref{eq:convergence-fundamental-condition-iter-h}, we obtain
    \begin{multline}
        \label{eq:convergence-fundamental-condition-iter-h-transformed0}
        \frac{1}{2}\norm{\nextu-\thisu}_{\Test_{i+1} \Precond_{i+1}}^2
        + \inf_{u^* \in \inv H(0)}\left(
            \frac{1}{2}\norm{\nextu-u^*}_{\Test_{i+1}\Precond_{i+1}}^2
            - \iprod{\nextu-\thisu}{\nextu-u^*}_{\Test_{i+1}\Precond_{i+1}}
        \right)
        \\
        \ge
        \frac{1}{2}\dist^2_{\Test_{i+2}\Precond_{i+2}}(\nextu; \inv H(0)).
    \end{multline}
    We recall for general self-adjoint $M$ the three-point formula
    \begin{equation}
       \label{eq:standard-identity}
       \iprod{\nextu-\thisu}{\nextu-u^*}_{M}
       = \frac{1}{2}\norm{\nextu-\thisu}_{M}^2
           - \frac{1}{2}\norm{\thisu-u^*}_{M}^2
           + \frac{1}{2}\norm{\nextu-u^*}_{M}^2.
    \end{equation}
    Using this  with $M=\Test_{i+1}\Precond_{i+1}$, we rewrite \eqref{eq:convergence-fundamental-condition-iter-h-transformed0} as
    \[
        \frac{1}{2}\dist^2_{\Test_{i+1}\Precond_{i+1}}(\thisu; \inv H(0))
        \ge
        \frac{1}{2}\dist^2_{\Test_{i+2}\Precond_{i+2}}(\nextu; \inv H(0)).
    \]
    Summing over $i=0,\ldots,n-1$, we obtain the claim.
\end{proof}

In the present work, we concentrate on $\Test_{i+1}\Precond_{i+1} \ge 0$ growing fast enough that \eqref{eq:convergence-result-main-h} yields convergence rates. If such growth is not present, then under some technical assumptions, it is still possible to obtain weak convergence. Since this is not our focus, we merely refer to \cite{tuomov-proxtest} for arguments applicable when $\dist_{\Test_{N+1}\Precond_{N+1}}^2(u^N, \inv H(0))$ is replaced by $\norm{u^N-\realoptu}_{\Test_{N+1}\Precond_{N+1}}^2$ for some fixed $\realoptu \in \inv H(0)$.

\subsection{Submonotonicity}
\label{sec:rates-from-monotonicity}

Suppose $\Test_{i+1} H$ is strongly monotone with respect to the operator $\Test_{i+1}\Xi_{i+1}$ for some $\Xi_{i+1} \in \linear(\Space; \Space)$ in the sense that
\begin{equation}
    \label{eq:abstractmono}
    \iprod{H(u)-H(u')}{u-u'}_{\Test_{i+1}} \ge \frac{1}{2}\norm{u-u'}^2_{\Test_{i+1}\Xi_{i+1}}
    \quad (u, u' \in \Space).
\end{equation}
Then, by taking $u'=u^* \in \inv H(0)$, we see \eqref{eq:convergence-fundamental-condition-iter-h} to hold if
\[
    \frac{1}{2}\norm{\nextu-\thisu}_{\Test_{i+1} \Precond_{i+1}}^2
    + \inf_{u^* \in \inv H(0)}
        \frac{1}{2}\norm{\nextu-u^*}_{\Test_{i+1}(\Precond_{i+1}+\Xi_{i+1})}^2
    \ge
    \frac{1}{2}\dist^2_{\Test_{i+2}\Precond_{i+2}}(\nextu; \inv H(0)).
\]
By the definition of $\dist^2_{\Test_{i+2}\Precond_{i+2}}$, this holds if we secure $\Test_{i+1}(\Precond_{i+1}+\Xi_{i+1}) \ge \Test_{i+2}\Precond_{i+2}$.
This limits the growth of $i \mapsto \Test_{i+1}\Precond_{i+1}$ through the available strong monotonicity, and hence the convergence rates \eqref{eq:convergence-result-main-h} can yield.

We do not, however, need the full power of strong monotonicity. Indeed, since $u'=u^* \in \inv H(0)$ was arbitrary above, we are led to think we can take the infimum in \eqref{eq:abstractmono} over $u' \in \inv H(0)$. However, we have to be careful to keep this infimisation compatible with $\dist^2_{\Test_{i+2}\Precond_{i+2}}(\nextu; \inv H(0))$. We therefore introduce the following concept.

\begin{definition}
    \label{def:submonotonicity}
    Let $U$ be a Hilbert space and $N, M, \Xi \in \linear(\Space; \Space)$ with $M \ge 0$.
    We say that $T: \Space \setto \Space$ is $(\Xi, N, M)$-\term{partially strongly submonotone} \emph{at $(\realoptu, \realoptw) \in \graph T$} (or \emph{at $\realoptu$ for $\realoptw \in T(\realoptu)$}) if there exists a neighbourhood $\mathcal{U} \ni \realoptu$ with
    \begin{equation}
        \label{eq:submonotonicity}
        \tag{PSM}
        \inf_{u^* \in \inv T(\realoptw)}
        \left(
        \iprod{w-\realoptw}{u-u^*}_N
        + \norm{u-u^*}_{M-\Xi}^2
        \right)
        \ge \dist^2_M(u, \inv T(\realoptw))
        \quad ( u \in \mathcal{U},\, w \in T(u)).
    \end{equation}
    If $\Xi=M$, we say that $T$ is $(N,M)$-\term{strongly submonotone}.
    If $\Xi=0$, we say that $T$ is $(N,M)$-\term{submonotone}.
\end{definition}

The idea is that $\Xi$ is the “part” of $M$ with respect to which $T$ is \emph{strongly} submonotone. 
If we fix the element $u^* \in \inv T(\realoptw)$ in both the infimisations and the $\dist$ in  \eqref{eq:submonotonicity}, then $M$ cancels out and the inequality reduces to a variant of the operator-relative strong monotonicity \eqref{eq:abstractmono}. The infimisation is the reason for calling the concept \emph{sub}monotonicity.

\begin{remark}[Submonotonicity from local monotonicity]
    \label{rem:monotonicity-to-submonotonicity}
    $(\Xi,M,N)$-partial strong submonotonicity for any $M \ge 0$ is implied by
    \begin{equation*}
        \iprod{w-\realoptw}{u-u^*}_N
        \ge \norm{u-u^*}_{\Xi}^2
        \quad ( u \in \mathcal{U},\, w \in T(u),\, u^* \in \inv T(\realoptw)).
    \end{equation*}
\end{remark}

\begin{remark}[Limited dependence on base point]
    Submonotonicity only depends on $\realoptu$ through $\mathcal{U}$.
\end{remark}

Returning to the preconditioned proximal point method \eqref{eq:pp}, the next result shows how $\Test_{i+2}\Precond_{i+2}$ can be made to grow based on partial strong submonotonicity, and therefore how this can help us obtain convergence rates.

\begin{theorem}
    \label{thm:convergence-result-submonotone}
    On a Hilbert space $\Space$, let $H: \Space \setto \Space$, and $\Precond_{i+1}, \Test_{i+1}, \Xi_{i+1} \in \linear(\Space; \Space)$ with $\Test_{i+1}\Precond_{i+1} \ge 0$ self-adjoint for all $i \in \N$.
    Suppose $\inv H(0)$ is non-empty and \eqref{eq:pp} is solvable for the iterates $\{\nextu\}_{i \in \N}$ given $u^0 \in \Space$. If $H$ is $(\Test_{i+1}\Xi_{i+1}, 2\Test_{i+1}, \Test_{i+2}\Precond_{i+2})$-partially strongly submonotone at some $(\realoptu, 0) \in \graph H$ with a uniform neighbourhood $\mathcal{U}$ for all $i \in \N$, and
    \begin{equation}
        \label{eq:convergence-fundamental-condition-iter-h-submonotonicity}
        \Test_{i+1}(\Precond_{i+1}+\Xi_{i+1}) \ge \Test_{i+2}\Precond_{i+2}
        \quad (i \in \N),
    \end{equation}
    then \eqref{eq:convergence-result-main-h} holds provided $\{\thisu\}_{i=0}^N \subset \mathcal{U}$ for the neighbourhood $\mathcal{U}$ of partial strong submonotonicity.
\end{theorem}

\begin{proof}
    By the assumed partial strong submonotonicity, for all $u^* \in \inv H(0)$, we have
    \[
        \iprod{H(\nextu)}{\nextu-u^*}_{\Test_{i+1}}
        + \frac{1}{2}\norm{\nextu-u^*}_{\Test_{i+2}\Precond_{i+2}-\Test_{i+1}\Xi_{i+1}}
        \ge
        \frac{1}{2}\dist^2_{\Test_{i+2}\Precond_{i+2}}(\nextu; \inv H(0)).
    \]
    Using \eqref{eq:convergence-fundamental-condition-iter-h-submonotonicity} and taking the infimum over $u^* \in \inv H(0)$, we obtain \eqref{eq:convergence-fundamental-condition-iter-h}.
    Then we just apply \cref{lemma:main-convergence}.
\end{proof}

\begin{remark}
    If the testing operators $\Test_{k+1} \equiv I$, then \cref{thm:convergence-result-submonotone} requires $(\Xi_{i+1}, 2I, \Precond_{i+2})$-partial strong submonotonicity, roughly speaking, submonotonicity with respect to the the part $\Xi_{i+1}/2$ of $\Precond_{i+2}/2$.
\end{remark}

\subsection{Transformation of submonotonicity}
\label{sec:trans-submono}

Before looking at more concrete examples, we provide results to transform submonotonicity with respect to one triple of operators to another triple of operators. The first result is basic:

\begin{proposition}[Scaling invariance]
    \label{prop:scaling-invariance}
    On a Hilbert space $\Space$, suppose $T: \Space \setto \Space$ is $(\Xi, N, M)$-partially strongly submonotone at $(\realoptu, \realoptw) \in \graph T$ in a neighbourhood $\mathcal{U}$. Then for any $\alpha>0$, at the same point $(\realoptu, \realoptw)$ in the same neighbourhood $\mathcal{U}$,
    \begin{enumerate}[label=(\roman*),noitemsep]
        \item\label{item:scaling-inveriance:basic} $T$ is $(\alpha\Xi, \alpha N, \alpha M)$-partially strongly submonotone.
        \item\label{item:scaling-inveriance:m} $T$ is $(\Xi, N, (1+\alpha)M)$-partially strongly submonotone.
    \end{enumerate}
\end{proposition}

\begin{proof}
    \cref{item:scaling-inveriance:basic} is obvious.
    For \cref{item:scaling-inveriance:m} using the $(\Xi, N, M)$-partial strong submonotonicity we obtain
    \begin{multline*}
        \inf_{u^* \in \inv T(\realoptw)}
        \left(
        \iprod{w-\realoptw}{u-u^*}_N
        + \norm{u-u^*}_{(1+\alpha)M-\Xi}^2
        \right)
        \\
        \ge
        \inf_{u^* \in \inv T(\realoptw)}
        \left(
        \iprod{w-\realoptw}{u-u^*}_N
        + \norm{u-u^*}_{M-\Xi}^2
        \right)        
        +
        \alpha \inf_{u^* \in \inv T(\realoptw)}  \norm{u-u^*}_{M}^2
        \\
        \ge
        \dist^2_{(1+\alpha)M}(u, \inv T(\realoptw))
        \quad ( u \in \mathcal{U},\, w \in T(u)).
        \qedhere
    \end{multline*}
\end{proof}

The following results will require the following concept, which will also be central \cref{sec:relationships}:

\begin{definition}
    On a Hilbert space $\Space$, for linear operators $M \ge M' \in \linear(\Space; \Space)$, a set $A \subset \Space$, and a point $\realoptu \in U$, we write $(M, M') \in \mathcal{P}(A, \realoptu)$ if there exists a neighbourhood $\mathcal{U}' \ni \realoptu$ such that each $u \in \mathcal{U}'$ has a common projection $u^*$ to the set $A$ with respect to both of the norms $\norm{\freevar}_{M}$ and $\norm{\freevar}_{M'}$:
    \[
        \argmin_{u' \in A} \norm{u-u'}_M=u^*=\argmin_{u' \in A} \norm{u-u'}_{M'}.
    \]
\end{definition}

\begin{example}[Unique solution on a subspace]
    Let $H: \Space \setto \Space$ with $U=X \times Y$. Write $u=(x, y)$.
    Take $M \ge \Xi \defeq \left(\begin{smallmatrix} \alpha I & 0 \\ 0 & 0 \end{smallmatrix}\right)$ for some $\alpha>0$.
    If $u^*=(x^*, y^*) \in \inv H(0)$ implies $x^*=\realoptx$ for some fixed $\realoptx$, then $(M, M-\Xi) \in \mathcal{P}(\inv H(0), \realoptu)$.
\end{example}

With this, we can show that strong submonotonicity implies partial strong submonotonicity for a large class of $M$.

\begin{proposition}
    \label{prop:submonotonicity-vs-partial}
    On a Hilbert space $\Space$, let $T: \Space \setto \Space$, and  $\Xi, N, M \in \linear(\Space; \Space)$ with $M \ge \Xi \ge 0$.

    If $(\Xi, M-\Xi) \in \mathcal{P}(\inv T(\realoptw), \realoptu)$, then $(N, \Xi)$-strong submonotonicity at $(\realoptu, \realoptw) \in \graph T$ implies $(\Xi, N, M)$-partial strong submonotonicity at the same point.

    If $\inv T(\realoptw)=\{\realoptu\}$ is a singleton, these two properties are equivalent; moreover $(\Xi, M-\Xi) \in \mathcal{P}(\inv T(\realoptw), \realoptu)$ automatically holds.
\end{proposition}

\begin{proof}
    The condition \eqref{eq:submonotonicity} for $(N, \Xi)$-strong submonotonicity reads
    \begin{equation}
        \label{eq:submonotonicity-vs-partial-1}
        \inf_{u^* \in \inv T(\realoptw)}
        \iprod{w-\realoptw}{u-u^*}_N
        \ge \dist^2_\Xi(u, \inv T(\realoptw))
        \quad ( u\in \mathcal{U},\, w \in T(u)).
    \end{equation}
    Since $(\Xi, M-\Xi) \in \mathcal{P}(\inv T(\realoptw), \realoptu)$, we have
    \[
        \dist^2_\Xi(u, \inv T(\realoptw))+\dist^2_{M-\Xi}(u, \inv T(\realoptw))
        \ge \dist^2_M(u, \inv T(\realoptw))
        \quad (u \in \mathcal{U}').
    \]
    Adding $\dist^2_{M-\Xi}(u, \inv T(\realoptw))$ on both sides of \eqref{eq:submonotonicity-vs-partial-1}, we therefore obtain
    \begin{equation*}
        \inf_{u^* \in \inv T(\realoptw)}
        \iprod{w-\realoptw}{u-u^*}_N
        +\dist^2_{M-\Xi}(u, \inv T(\realoptw))
        \ge \dist^2_M(u, \inv T(\realoptw))
        \quad (u \in \mathcal{U} \isect \mathcal{U}',\, w \in T(u)).
    \end{equation*}
    Using $\inf + \inf \le \inf$ on the left hand side, we $(\Xi, N, M)$-partial strong submonotonicity following the defining \eqref{eq:submonotonicity}.

    If $\inv T(\realoptw)=\{\realoptu\}$, retracing the steps above using the uniqueness of $u^*=\realoptu \in \inv T(\realoptw)$ establishes the equivalence claim.
\end{proof}





The next example shows that if $\inv T(\realoptw)$ is not a singleton, the converse implication may not hold.

\begin{example}
    \label{ex:nonimplication}
    Let $f: \R \to \R$, $f(u) \defeq \tfrac{1}{2}\min\{\abs{u+1}^2,\abs{u-1}^2\}$.
    Then $\subdiff_C f$ is $(\gamma I, I, I)$-partially strongly submonotone for any $\gamma \in (0, 1)$ at $(\realoptu, 0)$ when $\realoptu \in \{1,-1\}$, but is not  $(I, \gamma I)$-strongly submonotone.
\end{example}

\begin{demonstration}
    Clearly $\inv{[\subdiff f_C]}(0)=\{-1,0,1\}$.
    We take for simplicity $\realoptu=1$ and $\mathcal{U} \subset (0, \infty)$. The case $\realoptu=-1$ is similar. Then $\subdiff_C f(u)=\{u-1\}$ for $u \in \mathcal{U}$.
    For the $(\gamma I, I, I)$-partial strong submonotonicity, \eqref{eq:submonotonicity} consequently expands as
    \[
        \min_{u^* \in \{-1,0,1\}}\bigl(
        (u-1)(u-u^*)
        + (1-\gamma)\abs{u-u^*}^2
        \bigr)
        \ge \min\{\abs{u-1}^2,\abs{u}^2\}
        \quad (u \in \mathcal{U}).
    \]
    For $u^*=1$ the inequality is obvious. For the rest, by Young's inequality, it suffices if
    \[
        \min_{u^* \in \{-1,0\}}
        (1-\gamma-\inv\alpha)(u-u^*)^2
        \ge
        (1+\alpha/4)(u-1)^2
        \quad (u \in \mathcal{U})
    \]
    for some $\alpha>0$.
    Taking $\gamma \in (0, 1)$, $\alpha > 1/(1-\gamma)$, and a correspondingly small neighbourhood $\mathcal{U}$ around $\realoptu=1$, we prove the partial strong submonotonicity.

    For the $(\gamma I, I)$-strong submonotonicity we would need
    \[
        \min_{u^* \in \{-1,0,1\}}
        (u-1)(u-u^*),
        \ge \gamma\min\{\abs{u-1}^2,\abs{u}^2\}
        \quad (u \in \mathcal{U}).
    \]
    The left-hand side is negative for $u<1$, so no neighbourhood $\mathcal{U}$ of $\realoptu=1$ satisfies the condition.
\end{demonstration}

\subsection{Examples: the basic proximal point method}
\label{sec:prox-submono}

We next provide examples regarding the basic proximal point method. We return to primal-dual methods in \cref{sec:saddle}.
We first consider the non-operator case of partial strong submonotonicity, and afterwards use the full power of the operator formulation of partial strong submonotonicity to study convergence on subspaces.

\begin{example}[Basic proximal point method, submonotonicity]
    \label{ex:prox-submonotonicity}
    Suppose for some $m > \gamma >0$ that $H: \Space \setto \Space$ is $(\gamma I, 2 I, m I)$-partially strongly submonotone at $(\realoptu, 0) \in \graph H$ in the neighbourhood $\mathcal{U}$.
    Let $\{\nextu\}_{i \in \N}$ be generated by the basic proximal point method
    \begin{equation}
        \label{eq:basicprox}
        0 \in H(\nextu) + \inv\tau(\nextu-\thisu)
    \end{equation}
    for some $u^0 \in \Space$ and a step length parameter $\tau \ge 1/(m-\gamma)$. If $\{\thisu\}_{i \ge n} \in \mathcal{U}$ for large enough $n$, then $\dist^2(u^i; \inv H(0)) \to 0$ at a linear rate.
\end{example}

\begin{demonstration}
    We take  $\Precond_{i+1} \equiv \inv\tau I$, $\Xi_{i+1} \defeq \gamma I$, and $\Test_{i+1} \defeq \tauTest_i I$ for $\tauTest_{i+1} \defeq \tauTest_i(1+\gamma\tau)$, $\tauTest_0 \defeq 1$.
    Then \eqref{eq:pp} describes \eqref{eq:basicprox}, \eqref{eq:convergence-fundamental-condition-iter-h-submonotonicity} holds as an equality, and the $(\Test_{i+1}\Xi_{i+1}, 2\Test_{i+1}, \Test_{i+2}\Precond_{i+2})$-partial strong submonotonicity required by \cref{thm:convergence-result-submonotone} is by \cref{prop:scaling-invariance}\,\cref{item:scaling-inveriance:basic} equivalent to $(\gamma I, 2 I, (\inv\tau+\gamma) I)$-partial strong submonotonicity. 
    By \cref{prop:scaling-invariance}\,\cref{item:scaling-inveriance:m} this follows from the assumed $(\gamma I, 2 I, m I)$-partial strong submonotonicity when $\inv\tau+\gamma \ge m$, that is, $\tau \ge 1/(m-\gamma)$.
    Thus by \cref{thm:convergence-result-submonotone}, $\dist^2(u^N; \inv H(0)) \le (\tauTest_0/\tauTest_N)\dist^2(u^0; \inv H(0))$ for large enough $N$. This yields the claimed convergence rate.
\end{demonstration}

\begin{remark}[Superlinear convergence]
    With iteration-dependent step lengths $\tau_k \upto \infty$, it is also possible to obtain superlinear convergence. Under basic strong monotonicity this is proved in \cite{rockafellar1976monotone}; in the “testing” framework, which is also our overall approach to convergence proofs, in \cite[Example 2.2]{tuomov-proxtest} or \cite[Theorem 10.1]{clasonvalkonen2020nonsmooth}.
\end{remark}

\begin{remark}[Reaching the local neighbourhood]
    \label{rem:localtoglobal}
    If $\Space$ is finite-dimensional, the condition $\{\thisu\}_{i \ge n} \in \mathcal{U}$ can be guaranteed by standard convergence results \cite{martinet1970regularisation} for large $n \in \N$ and some $\realoptu \in \inv H(0)$ with the corresponding neighbourhood $\mathcal{U}$.
\end{remark}

A basic application of \cref{ex:prox-submonotonicity} is to $H=\subdiff G$ for a convex function $G$.
As basic building blocks, we next show on $\R$ that the subdifferentials of the indicator of the unit ball, and of the absolute value function are strongly submonotone. None of these subdifferentials are \emph{strongly} monotone in the conventional sense. With $(\realoptx, \realoptq) \in \graph \subdiff G$, we will find a neighbourhood $\mathcal{U} \ni \realoptu$ such that we have the $(I, \gamma I)$-strong submonotonicity
\begin{equation}
    \label{eq:convex-submono}
    \inf_{x^* \in \inv{[\subdiff G]}(\realoptq)} \iprod{q-\realoptq}{x-x^*} \ge
    \gamma  \inf_{x^* \in \inv{[\subdiff G]}(\realoptq)} \norm{x-x^*}^2
    \quad (x \in \mathcal{U},\, q \in \subdiff G(x)).
\end{equation}
We recall from \cref{prop:submonotonicity-vs-partial} that $(I, \gamma I)$-strong submonotonicity, i.e., $(\gamma I, I, \gamma I)$-partial strong submonotonicity, implies $(\gamma I, I, I)$-partial strong submonotonicity,and by \cref{prop:scaling-invariance}, further $(2\gamma I, 2I, 2m I)$-partial strong submonotonicity for any $m > 1$.

\begin{lemma}
    \label{lemma:subreg-ball-indicator}
    Let $G \defeq \delta_{\closure \B(0, \alpha)}$ on $\R^n$. Then $\subdiff G$ is $(I, \gamma I)$-strongly submonotone at $(\realoptx, \realoptq) \in \graph \subdiff G$ with
    \[
        \mathcal{U} \defeq \Dom G
        \quad\text{and}\quad
        \gamma \defeq
            \begin{cases}
                \norm{\realoptq}/(2\alpha), & \realoptq \ne 0, \\
                \infty, & \realoptq=0.
            \end{cases}
    \]
\end{lemma}

\begin{proof}
    We need to prove \eqref{eq:convex-submono}.
    If $\realoptq=0$, then $\inv{[\subdiff G]}(\realoptq)=\closure \B(0, \alpha)$, so  \eqref{eq:convex-submono} trivially holds by the monotonicity of $\subdiff G$ as a convex subdifferential \cite{rockafellar-convex-analysis}.

    Otherwise, if $\realoptq \ne 0$, necessarily $\realoptq=\beta \realoptx$ for some $\beta>0$, and $\norm{\realoptx}=\alpha$. Moreover, $\inv{[\subdiff G]}(\realoptq)=\{\realoptx\}$, and $\iprod{-q}{\realoptx-x} \ge G(x)-G(\realoptx)=0$.
    Therefore \eqref{eq:convex-submono} holds if $\beta \iprod{\realoptx}{\realoptx-x} \ge \gamma \norm{x-\realoptx}^2$.
    In other words $
        (\beta-\gamma)\norm{\realoptx}^2
        \ge \gamma\norm{x}^2 + (\beta-2\gamma)\iprod{\realoptx}{x}$.
    Clearly $x \in \Dom \subdiff G$ implies $\norm{x} \le \alpha$.
    Since  $\norm{\realoptx}=\alpha$, this condition holds for $\beta \ge 2\gamma$. Since $\realoptq=\beta \realoptx$ and $\norm{\realoptx}=\alpha$, the maximal choice is $\gamma=\norm{\realoptq}/(2\alpha)$.
\end{proof}

\begin{lemma}
    \label{lemma:subreg-1norm}
    Let $G \defeq \abs{\freevar}$ on $\R$.
    Then $\subdiff G$ is $(I, \gamma I)$-strongly submonotone at $(\realoptx, \realoptq) \in \graph \subdiff G$ for any $\gamma>0$ in the neighbourhood
    \[
        \mathcal{U} \defeq
        \begin{cases}
            [-2\inv \gamma, \infty), & \realoptq=1 \\
            (-\infty, 2\inv \gamma], & \realoptq=-1, \\
            ([-1,1]-\realoptq)\inv\gamma, & \abs{\realoptq} < 1.
        \end{cases}
    \]
\end{lemma}

\begin{proof}
	We need to prove \eqref{eq:convex-submono}.
    We have $\realoptq \in [-1, 1]$. Consider first $\realoptq=1$.
    Then $\inv{[\subdiff G]}(\realoptq)=[0,\infty)$, so  \eqref{eq:convex-submono} reduces to
    \[
        \inf_{x^* \ge 0,\, q \in \subdiff G(x)} (q-1)(x-x^*) \ge \gamma \inf_{x^* \ge 0} \abs{x-x^*}^2
    \]
    If $x \ge 0$, this holds for any $\gamma \ge 0$ by the monotonicity of $\subdiff G$. Otherwise, if $x < 0$, we have $q = -1$, so need $-2x \ge \gamma\abs{x}^2$, which, because $x < 0$, is to say $2 \ge \gamma\abs{x}$. This is guaranteed by our choice of $\mathcal{U}$.
    
    The case $\realoptq=-1$ is analogous.

    Consider then $\abs{\realoptq} < 1$. Then $\inv{[\subdiff G]}(\realoptq)=\{\realoptx\}=\{0\}$, so \eqref{eq:convex-submono} holds if
    \begin{equation}
    	\label{eq:subreg-1norm-case2}
        \inf_{q \in \subdiff G(x)} (q-\realoptq)x \ge \gamma x^2.
    \end{equation}
    If $x=0$, this is clear. If $x>0$, $\subdiff G(x)=\{1\}$, so \eqref{eq:subreg-1norm-case2} holds if $1-\realoptq \ge \gamma x$. This holds if $x \le (1-\realoptq)/\gamma$. Similarly, if $x<0$, we obtain for \eqref{eq:subreg-1norm-case2} condition $-1-\realoptq \le \gamma x$. This holds when $x \ge (-1-\realoptq)/\gamma$.
    The conditions $x \le (1-\realoptq)/\gamma$ and $-1-\realoptq \le \gamma x$ give the expression for $\mathcal{U}$ in the statement of the lemma.
\end{proof}

\begin{example}
    Let $F(x, y) \defeq \abs{x}+\delta_{[-1, 1]}(y)$.
    Then $\subdiff F$ is $(I, \gamma I)$-strongly submonotone at $((\realoptx, \realopty), 0)$ for any $(\realoptx, \realopty) \in \{0\} \times [-1, 1]$ and $\gamma>0$.
    Consequently the iterates of the proximal point method of \cref{ex:prox-submonotonicity} for $H=\subdiff F$ and any $\tau>0$ converge linearly to $\{0\} \times [-1, 1]$ 
\end{example}

\begin{demonstration}
    Clearly $\{0\} \times [-1, 1] = \inv H(0)$.
    We combine \cref{lemma:subreg-1norm,lemma:subreg-ball-indicator} to show the $(I, \gamma I)$-strong submonotonicity. By \cref{prop:scaling-invariance,prop:submonotonicity-vs-partial} we then obtain $(2\gamma I, 2I, 2m I)$-partial strong submonotonicity for any $m > 1$. By taking $m > \gamma$ large enough that $\tau \ge 1/(2m-2\gamma)$, we obtain the claim from \cref{ex:prox-submonotonicity} provided the iterates are in the neighbourhood of submonotonicity $\mathcal{U}$.
    By \cref{rem:localtoglobal} the convergence is, in fact, global.
\end{demonstration}

\subsection{Examples: varying behaviour on subspaces}
\label{sec:prox-submono-partial}

The next examples demonstrate the word ``partial'' in the definitions.

\begin{example}[Partial convergence of a preconditioned proximal point method]
    Suppose $H: X \times Y \setto X \times Y$ is $(\Xi, 2I, M)$-partially strongly submonotone at $(\realoptu, 0) \in \graph H$ in the (fixed) neighbourhood $\mathcal{U}$ with $\Xi=\diag(\gamma I, 0)$ for some fixed $\gamma>0$ and $M=\diag(m_x I, m_y I)$ for any $m_x > \gamma$ and $m_y > 0$.
    For any $\tau,\sigma_0>0$ and initial $u^0=(x^0, y^0) \in X \times Y$, let $\{\nextu=(\nextx,\nexty)\}_{i \in \N}$ be generated by the preconditioned proximal point method
    \[
        0 \in H(\nextu) + M_{i+1}(\nextu-\thisu)
        \quad\text{where}\quad
        M_{i+1} = \begin{pmatrix}\inv\tau I & 0 \\ 0 & \inv\sigma_i I\end{pmatrix}
        \quad\text{and}\quad
        \sigma_{i+1} \defeq (1+\gamma\tau)\sigma_i.
    \]
    If  $\{\thisu\}_{i \ge n} \in \mathcal{U}$ for large enough $n$, then $\dist(\thisx, \{\realoptx \mid (\realoptx, \realopty) \in \inv H(0)\}) \to 0$ at a linear rate.
\end{example}

\begin{demonstration}
    We take $\Test_{i+1} = \tauTest_i I$ and $\Xi_{i+1} = \diag(\gamma I, 0)$ for
    $\tauTest_{i+1} \defeq \tauTest_i(1+\gamma\tau)$ and $\tauTest_0 \defeq 1$.
    Then $\sigma_{i+1} = \sigma_i\tauTest_{i+1}\inv\tauTest_i$.
    We use \cref{thm:convergence-result-submonotone}.
    The condition \eqref{eq:convergence-fundamental-condition-iter-h-submonotonicity} holds with our choices.
    Moreover,
    \[
        \Test_{i+2}\Precond_{i+2}
        =
        \begin{pmatrix}
            \tauTest_i(\inv\tau+\gamma) I  &  0
            \\
            0 & \tauTest_i\inv\sigma_i I
        \end{pmatrix}
        \quad\text{and}\quad
        \Test_{i+1}\Xi_{i+1}
        =
        \begin{pmatrix}
            \tauTest_i\gamma I & 0
            \\
            0 & 0
        \end{pmatrix}.
    \]
    Thus the required $(\Test_{i+1}\Xi_{i+1}, 2\Test_{i+1}, \Test_{i+2}\Precond_{i+2})$-partial strong submonotonicity amounts to
    \[
        (\diag(\tauTest_i\gamma I, 0), 2\tauTest_i I, \diag(\tauTest_i(\inv\tau+\gamma)I, \tauTest_i\inv\sigma_i))\text{-partial strong submonotonicity}.
    \]
    By the scaling invariance \cref{prop:scaling-invariance}, this is implied by our assumptions with $m_x=\inv\tau+\gamma$ and $m_y=\inv\sigma_i$.
    Now $\{\tauTest_i\}_{i \in \N}$ grows exponentially, so  \eqref{eq:convergence-result-main-h} established by \cref{thm:convergence-result-submonotone} shows the claim.
\end{demonstration}

\begin{figure}
    \centering
    \begin{asy}
        settings.render=0;
        import graph3;
        import contour;
        import palette;
        size(150,0);
        currentprojection=perspective(3,-4,20);
        real mu=0.5;
        real f(pair p) {
            return (p.x^2+p.y^2)*p.y^2 + max(p.x-mu, max(-p.x-mu, 0));
        }
        //draw((-1,-1,0)--(1,-1,0)--(1,1,0)--(-1,1,0)--cycle);
        //draw(arc(0.12Z,0.2,90,60,90,25),ArcArrow3);
        surface s=surface(f,(-1,-1),(1,1),nx=20,Spline);
        s.colors(palette(s.map(zpart),Grayscale()+gray));
        xaxis3(Label("$x$"),black,Arrow3,xmin=0,xmax=1.5);
        yaxis3(Label("$y$"),black,Arrow3,ymin=0,ymax=1.5);
        //zaxis3(XYZero(extend=true),red,Arrow3);
        pen[] Palette=Grayscale();
        draw(s,meshpen=darkgray,light=nolight);//,render(merge=true));
        real[] levels={0,0.1,0.2,0.3};
        draw(lift(f,contour(f,(-1,-1),(1,1),levels)),1bp+red);
        //label("$O$",O,-Z+Y,red);
    \end{asy}
    \caption{The function $F$ from \cref{ex:partial} with some contour lines plotted. The function is submonotone at any $((\realoptx, 0), 0) \in \graph \subdiff F$ strongly in the $x$ direction, but non-strongly in the $y$ direction.}
    \label{fig:partial}
\end{figure}

\begin{example}[Varying behaviour along subspaces]
    \label{ex:partial}
    With $u=(x, y) \in \R^2$ and $\mu > 0$, let
    \[
        F(u) \defeq G(x) + \frac{1}{2}(x^2+y^2)y^2
        \quad\text{with}\quad
        G(x) \defeq \begin{cases}
            0, & \abs{x} \le \mu, \\
            x-\mu, & x > \mu, \\
            x+\mu, & x < - \mu.
            \end{cases}
    \]
    This is illustrated in \cref{fig:partial}.
    Let $\gamma>0$.
    Then $\subdiff F$ is $(\Xi, 2I, M)$-partially strongly submonotone at any $((\realoptx, 0), 0) \in \graph \subdiff f$ in the neighbourhood $\mathcal{U}=(\mu+2\inv\gamma)[-1,1] \times \R$ with $M=\diag(m_x I, m_y I)$ and $\Xi=\diag(\gamma I, 0)$ for any $m_x \ge \gamma$ and $m_y \ge 0$.
\end{example}

\begin{demonstration}   
    Let $Q(u) \defeq \frac{1}{2}(x^2+y^2)y^2$. The function $Q$ is convex as verified by the positivity of the Hessian, hence is $\grad Q$ monotone. Clearly $\inv{[\subdiff F]}(0)=[-\mu, \mu] \times \{0\}$ with $\grad Q(u^*)=0$ for $u^* \in \inv{[\subdiff F]}(0)$.    
    We use the notation $u=(x, y)$ and $u^*=(x^*, y^*)$. 
    We have $\subdiff F(u)=\{(w, 0)+\grad Q(u) \mid w \in \subdiff G(x)\}$ and, by the monotonicity of $\grad Q$,
    \[
        \inf_{u^* \in \inv{[\subdiff F]}(0)}
        \left(
        \iprod{(w, 0)+\grad Q(u)-0}{u-u^*}_{2I}
        + \norm{u-u^*}_{M-\Xi}^2
        \right)
        \ge
        \inf_{u^* \in \inv{[\subdiff F]}(0)}
        \left(
        2w(x-x^*)
        + \norm{u-u^*}_{M-\Xi}^2
        \right).
    \]
    Consequently, the claimed partial strong submonotonicity at $((\realoptx, 0), 0) \in \graph \subdiff F$ holds if
    \begin{equation}
        \label{eq:ps-ex-x}
        \min_{x^* \in [-\mu,\mu]}
        2w(x-x^*)+(m_x-\gamma)(x-x^*)^2 + m_y y^2 
        \ge \min_{x^{**} \in [-\mu,\mu]} m_x(x-x^{**})^2 +  m_y y^2
    \end{equation}%
    for all $u=(x, y)$ in a neighbourhood $\mathcal{U}$ of $\realoptu=(\realoptx, 0)$, and all
    \[
        w \in \subdiff G(x)= \begin{cases}
            0, & \abs{x} < \mu, \\
            1, & x > \mu, \\
            [0, 1], & x=\mu, \\
            -1, & x < - \mu, \\
            [-1, 0], & x=-\mu.
            \end{cases}
    \]

    If $x \ge \mu$, then $w \ge 1$ and $x>x^*$. With $m_x \ge \gamma$, clearly both sides of the inequality in \eqref{eq:ps-ex-x} have the same minimiser $x^*=x^{**}=\mu$. So \eqref{eq:ps-ex-x} holds if $2(x-\mu) \ge \gamma(x-\mu)^2$, which is to say with $x \le \mu + 2\inv\gamma$.
    
    The case $x \le -\mu$ is analogous.

    If $\abs{x} < \mu$, we have $w=0$, so with $m_x \ge \gamma$, both sides of \eqref{eq:ps-ex-x} are minimised and are equal with $x^{**}=x^{*}=x$.
\end{demonstration}

\section{Rates from error bounds and subregularity}
\label{sec:rates-from-error-bounds}

We now study an approach alternative to submonotonicity: the error bounds that we discussed in the introduction.
Their essence is to prove for some $\kappa >0$ that
\begin{equation*}
    \label{eq:basic-error-bound}
    \tag{EB}
    \kappa \norm{\nextu-\thisu} \ge \norm{\nextu-\realoptu}.
\end{equation*}
One can see how this would improve \eqref{eq:convergence-fundamental-condition-iter-h} by allowing $\Test_{i+1}\Precond_{i+2}$ to grow faster. However, we generally cannot fix $\realoptu$, so would take the infimum over $\realoptu \in \inv H(0)$ above. We also want a \emph{partial}, operator-relative, variant, to deal with primal-dual methods, and even block-adapted methods \cite{tuomov-blockcp,tuomov-nlpdhgm-block}.

\subsection{Partial error bounds and subregularity}
\label{sec:error-bounds-from-subregularity}

We generalise error bounds to be operator-relative:

\begin{theorem}
    \label{thm:convergence-result-main-h-peb}
    On a Hilbert space $\Space$, let $H: \Space \setto \Space$, and $\Precond_{i+1}, \Test_{i+1}, P_{i+1} \in \linear(\Space; \Space)$ with $\Test_{i+1}\Precond_{i+1} \ge 0$ self-adjoint for all $i \in \N$.
    Suppose $\inv H(0)$ is non-empty and \eqref{eq:pp} is solvable for the iterates $\{\nextu\}_{i \in \N}$ given $u^0 \in \Space$.
    If for some $\delta \in [0, 1]$ the \term{partial error bound}
    \begin{equation}
        \label{eq:partial-error-bound}
        \tag{PEB}
        \delta\norm{\nextu-\thisu}^2_{\Test_{i+1}\Precond_{i+1}}
        +
        \dist_{\Test_{i+2}\Precond_{i+2}-\Test_{i+1}P_{i+1}}^2(\nextu, \inv H(0))
        \ge
        \dist_{\Test_{i+2}\Precond_{i+2}}^2(\nextu, \inv H(0))
    \end{equation}
    holds and
    \begin{equation}
        \label{eq:convergence-fundamental-condition-iter-h-peb}
        \frac{1-\delta}{2}\norm{\nextu-\thisu}_{\Test_{i+1} \Precond_{i+1}}^2
        + \frac{1}{2}\norm{\nextu-u^*}_{\Test_{i+1}(\Precond_{i+1}+P_{i+1})-\Test_{i+2}\Precond_{i+2}}^2
        + \iprod{H(\nextu)}{\nextu-u^*}_{\Test_{i+1}}
        \ge 0
    \end{equation}
    for all $i \in \N$ and $u^* \in \inv H(0)$, then \eqref{eq:convergence-result-main-h} holds.
\end{theorem}

\begin{proof}
    By \eqref{eq:convergence-fundamental-condition-iter-h-peb} for all $u^* \in \inv H(0)$ it holds
    \[
        \frac{\delta}{2}\norm{\nextu-\thisu}^2_{\Test_{i+1}\Precond_{i+1}}
        + \frac{1}{2}\norm{\nextu-u^*}_{\Test_{i+2}\Precond_{i+2}-\Test_{i+1}P_{i+1}}
        \ge
        \frac{1}{2}\dist^2_{\Test_{i+2}\Precond_{i+2}}(\nextu; \inv H(0)).
    \]
    Summing this with \eqref{eq:convergence-fundamental-condition-iter-h-peb}, and taking the infimum over $u^* \in \inv H(0)$, we obtain \eqref{eq:convergence-fundamental-condition-iter-h}.
    Then we just apply \cref{lemma:main-convergence}.
\end{proof}

An essential ingredient in proving the basic error bound \eqref{eq:basic-error-bound} is the \term{metric subregularity of $H$ at $\realoptu$ for $\realopt{w}=0$}: the existence of a neighbourhood $\mathcal{U} \ni \realoptu$ and $\kappa>0$ such that
\begin{equation}
    \label{eq:metric-subreg}
    \kappa \dist(\realopt{w}, H(u)) \ge \dist(u, \inv H(\realopt{w})) \quad (u \in \mathcal{U}).
\end{equation}
We refer to \cite{gfrerer2011first,ioffe2017variational,kruger2015error,dontchev2014implicit,ngai2008error} for more on error bounds and metric subregularity.
To prove \eqref{eq:partial-error-bound} we need a partial version.

\begin{definition}
    Let $\Space, \SpaceTwo$ be Hilbert spaces.
    Also let $M, P \in \linear(\Space; \Space)$ and $N \in \linear(\SpaceTwo; \SpaceTwo)$ with $N \ge 0$, $M \ge 0$, and $M \ge P$.
    We say say that $T: \Space \setto \SpaceTwo$ is \term{$(P,N,M)$-partially subregular} at $(\realoptu, \realoptw) \in \graph T$ if there exists a neighbourhood $\mathcal{U} \ni \realoptu$ such that
    \begin{equation}
        \label{eq:partial-subreg}
        \tag{PSR}
        \dist_N^2(\realopt{w}, T(u))
        + \dist_{M-P}^2(u, \inv T(\realopt{w}))
        \ge
        \dist_M^2(u, \inv T(\realopt{w}))
        \quad (u \in \mathcal{U}).
    \end{equation}
    We say that $T$ is $(N,M)$-subregular if $P=M$.
\end{definition}

\begin{lemma}
    \label{lemma:error-bound-first-estimate}
    Suppose $\Test_{i+1}\Precond_{i+1} \ge 0$ is self-adjoint and positive definite.
    Let $\nextu$ be generated by \eqref{eq:pp} given $\thisu$.
    Then
    \begin{equation}
         \label{eq:error-bound-first-estimate}
        \norm{\nextu-\thisu}_{\Test_{i+1}\Precond_{i+1}}^2
        \ge
        \dist_{\Test_{i+1}\inv{(\Test_{i+1}\Precond_{i+1})}\Test_{i+1}}^2(0, H(\nextu)).
    \end{equation}
\end{lemma}

\begin{proof}
    Let $\nexxt{q} \defeq -\Precond_{i+1}(\nextu-\thisu)$.
    Then $\nexxt{q} \in H(\nextu)$.
    By applying $\iprod{\freevar}{\nextu-\thisu}_{\Test_{i+1}}$ to \eqref{eq:pp}, we therefore obtain
    \[
        \norm{\nextu-\thisu}_{\Test_{i+1}\Precond_{i+1}}^2
        =
        -\iprod{\nexxt{q}}{\nextu-\thisu}_{\Test_{i+1}}.
    \]
    By our assumptions $\Test_{i+1}\Precond_{i+1}$ is invertible.
    Therefore we can solve $\nextu-\thisu=-\inv{(\Test_{i+1}\Precond_{i+1})}\Test_{i+1}\nexxt{q}$.
    It follows
    \[
         \norm{\nextu-\thisu}_{\Test_{i+1}\Precond_{i+1}}^2
         =
         \norm{\nexxt{q}}_{\Test_{i+1}^*\inv{(\Test_{i+1}\Precond_{i+1})}\Test_{i+1}}^2.
    \]
    This immediately yields the claim.
\end{proof}

As a consequence we obtain:

\begin{lemma}
    \label{lemma:subregularity-to-peb}
    Suppose $\Test_{i+1}\Precond_{i+1}$ is self-adjoint and positive definite.
    Let $P­_{i+1} \in \linear(\Space; \Space)$ and $\delta \in [0, 1]$.
    Then the partial error bound \eqref{eq:partial-error-bound} holds if $H$ is $(\Test_{i+1}P_{i+1},\delta \Test_{i+1}^*\inv{(\Test_{i+1}\Precond_{i+1})}\Test_{i+1},\Test_{i+2}\Precond_{i+2})$-partially subregular at some $\realoptu \in \inv H(0)$ in a neighbourhood $\mathcal{U}$  containing $\nextu$.
\end{lemma}

\begin{proof}
    The condition \eqref{eq:partial-subreg} now reads
    \[
        \dist_{\delta \Test_{i+1}^*\inv{(\Test_{i+1}\Precond_{i+1})}\Test_{i+1}}^2(0, H(u))
        + \dist_{\Test_{i+2}\Precond_{i+2}-\Test_{i+1}P_{i+1}}^2(u, \inv H(0))
        \ge
        \dist_{\Test_{i+2}\Precond_{i+2}}^2(u, \inv H(0))
        \quad (u \in \mathcal{U}).
    \]
    Taking $u=\nextu$ and combining with \eqref{eq:error-bound-first-estimate} from  \cref{lemma:error-bound-first-estimate} we obtain \eqref{eq:partial-error-bound}.
\end{proof}

\begin{corollary}
    \label{cor:convergence-result-subreg}
    On a Hilbert space $\Space$, let $H: \Space \setto \Space$,
    and $\Precond_{i+1}, \Test_{i+1},  P_{i+1} \in \linear(\Space; \Space)$ with $\Test_{i+1}\Precond_{i+1}$ self-adjoint for all $i \in \N$.
    Suppose $\inv H(0)$ is non-empty and that \eqref{eq:pp} is solvable for the iterates $\{\nextu\}_{i \in \N}$ given $u^0 \in \Space$.
    Pick $\delta \in [0, 1]$.
    If $H$ is $(\Test_{i+1}P_{i+1},\delta \Test_{i+1}^*\inv{(\Test_{i+1}\Precond_{i+1})}\Test_{i+1},\Test_{i+2}\Precond_{i+2})$-partially subregular at some $(\realoptu, 0) \in \graph H$ and \eqref{eq:convergence-fundamental-condition-iter-h-peb} holds (in particular, if $\Test_{i+1} H$ is monotone and
    \begin{equation}
        \label{eq:convergence-condition-subreg}
        \Test_{i+1}(\Precond_{i+1}+P_{i+1}) \ge \Test_{i+2}\Precond_{i+2} 
        \quad (i \in \N)),
    \end{equation}
    then \eqref{eq:convergence-result-main-h} holds provided $\{\thisu\}_{i=0}^N \subset \mathcal{U}$ for the neighbourhood $\mathcal{U}$ of partial subregularity.
\end{corollary}

\begin{proof}
    We use \cref{lemma:subregularity-to-peb} with \cref{thm:convergence-result-main-h-peb}.
\end{proof}


\begin{example}[Basic proximal point method, subregularity]
    \label{ex:prox-subregularity}
    Suppose $H: \Space \setto \Space$ is monotone and $(\pi I, \delta \tau^2 I, (1+\pi) I)$-partially subregular at $(\realoptu, 0) \in \graph H$ in the neighbourhood $\mathcal{U}$ for some $\tau >0$, $\pi \ge 0$, and $\delta \in [0, 1]$.
    Let $\{\nextu\}_{i \in \N}$ be generated by the basic proximal point method of \cref{ex:prox-submonotonicity} for some initial $u^0 \in \Space$.
    If $\{\thisu\}_{i \ge N} \subset \mathcal{U}$ for large enough $N \in \N$, then $\dist^2(u^i; \inv H(0)) \to 0$ at a linear rate.
\end{example}

\begin{demonstration}
    We take $\Precond_{i+1} \defeq \inv\tau I$, $P_{i+1} \defeq \pi I$, and $\Test_{i+1} \defeq \tauTest_i I$ for $\tauTest_{i+1} \defeq \tauTest_i(1+\pi)$ and $\tauTest_0 \defeq 1$.
    Then $\Test_{i+1}(\Precond_{i+1}+P_{i+1})=\Test_{i+2}\Precond_{i+2}$.  
    Combined with the monotonicity of $H$, this verifies \eqref{eq:convergence-fundamental-condition-iter-h-peb}.
    Observe then that $\delta \Test_{i+1}^*\inv{(\Test_{i+1}\Precond_{i+1})}\Test_{i+1}= \delta\tau^2\tauTest_i I$. By scaling invariance (detailed in the next \cref{prop:scaling-invariance:subreg}) we thus verify that $(\pi I, \delta \tau^2 I, (1+\pi) I)$-partial subregularity yields the subregularity demanded by \cref{lemma:subregularity-to-peb}. 
	Consequently we obtain the claim from  \cref{cor:convergence-result-subreg}.
\end{demonstration}

We do not in this instance provide further examples as they arise from those of \cref{sec:prox-submono,sec:prox-submono-partial} combined with the relationships between submonotonicity and subregularity that we study in the next section. They depend on the following transformation results.

\subsection{Transformation of subregularity}
\label{sec:trans-subreg}

The analogues of \cref{prop:scaling-invariance,prop:submonotonicity-vs-partial} hold for subregularity with virtually identical proofs.

\begin{proposition}
    \label{prop:scaling-invariance:subreg}
    Suppose $T: \Space \setto \Space$ is $(P, N, M)$-partially subregular at $(\realoptu, \realoptw) \in \graph T$ in a neighbourhood $\mathcal{U}$. Then for any $\alpha>0$, at the same point $(\realoptu, \realoptw)$ in the same neighbourhood $\mathcal{U}$,
    \begin{enumerate}[label=(\roman*),noitemsep]
        \item $T$ is $(\alpha P, \alpha N, \alpha M)$-partially subregular.
        \item $T$ is $(P, N, (1+\alpha)M)$-partially subregular.
    \end{enumerate}
\end{proposition}

\begin{proposition}
    \label{prop:subregularity-vs-partial}
    Let $T: \Space \setto \Space$, and  $P, N, M \in \linear(\Space; \Space)$ with $M \ge P \ge 0$.

    If $(P, M-P) \in \mathcal{P}(\inv T(\realoptw), \realoptu)$, then $(N, P)$-subregularity at $(\realoptu, \realoptw) \in \graph T$ implies $(P, N, M)$-partial subregularity at the same point.

    If $\inv T(\realoptw)=\{\realoptu\}$ is a singleton, these two properties are equivalent.
\end{proposition}




The next lemma helps scale some of the factors of partial subregularity in more general cases.

\begin{lemma}
    \label{lemma:subregularity-conversion}
    Let $T: \Space \setto \Space$, and  $P, N, M \in \linear(\Space; \Space)$ with $M \ge 0$, and $M \ge P$. Also pick $\alpha \ge 0$, and, if $\alpha \in (0, 1)$, suppose that $(M,M-P) \in \mathcal{P}(\inv T(\realoptw), \realoptu)$.
    Then $T$ is $(P,N,M)$-partially subregular at $(\realoptu,\realoptw) \in \graph T$ if it is $(\alpha P, \alpha N, M)$-partially subregular at this point.

    Conversely, if $(M,M-\alpha P) \in \mathcal{P}(\inv T(\realoptw), \realoptu)$ when $\alpha > 1$, then $T$ is $(\alpha P,\alpha N,M)$-partially subregular at $(\realoptu,\realoptw) \in \graph T$ if it is $(P, N, M)$-partially subregular at this point.

    In particular, if $(M,M-\max\{\alpha,1\} P) \in \mathcal{P}(\inv T(\realoptw), \realoptu)$, then the relationship is ``if and only if''.
\end{lemma}

\begin{proof}
    If $T$ is $(\alpha P,\alpha N,M)$-partially subregular, in some neighbourhood $\mathcal{U}$ of $\realoptu$ for every $u^* \in \inv T(\realoptw)$ and $u \in \mathcal{U}$ it holds
    \[
        \dist_{\alpha N}^2(\realoptw, T(u))
        + \norm{u-u^*}_{M-\alpha P}^2
        \ge
        \dist_M^2(u, \inv T(\realoptw)).
    \]
    After multiplying by $\inv\alpha$, we rearrange this as
    \[
        \dist_{N}^2(\realoptw, T(u))
        + \norm{u-u^*}_{M-P}^2
        \ge
        \dist_{\inv\alpha M}^2(u, \inv T(\realoptw))
        + \norm{u-u^*}_{M-\inv\alpha M}^2.
    \]
    If $\alpha<1$, we now let $u^*$ be the common projection of $u$ to the closed set $\inv T(\realopt{w})$ in the norms $\norm{\freevar}_{M}$ and $\norm{\freevar}_{M-P}$. Otherwise, if $\alpha \ge 1$, we just take the infimum over $u^* \in \inv T(\realoptw)$ on both sides.
    This proves $(P, N, M)$-partial subregularity.

    For the converse implication, we have to prove the relationship in the other direction. This amounts to applying the first claim with $\alpha$ replaced by $\inv\alpha$ to $(P', N',M)=(\alpha P, \alpha N, M)$. If $\alpha \in (0, 1]$, the projection condition is not required in this direction as $\inv\alpha \ge 1$; if $\alpha > 1$ then we need the projection to be the same with respect to $M$ and $M-P'=M-\alpha P$. This gives the projection condition in the converse claim.

    Finally, the ``if and only if'' claim just combines the two implications.
\end{proof}

\section{Relationships between subregularity and submonotonicity}
\label{sec:relationships}

Having introduced the distinct concepts of partial strong submonotonicity and partial subregularity motivated by algorithmic needs, we now study some basic theoretical relationships between these concepts. Throughout, we assume that $\Space$ and $\SpaceTwo$ are Hilbert spaces.

\subsection{General implications between submonotonicity and subregularity}

The following lemmas generalise \cite[Theorem 3.3]{aragon2008characterization} from convex subdifferentials to general set-valued maps and partial subregularity.

\begin{lemma}
    \label{lemma:submonotonicity-to-subregularity}
    With $\Space$ and $\SpaceTwo$ Hilbert spaces, let $T: \Space \setto \Space$, and $\Xi, N, M \in \linear(\Space; \Space)$ with $M \ge 0$.
    Pick $\alpha>0$.
    Suppose the following structural conditions hold:
    \begin{enumerate}[label=(\roman*)]
        \item
        \label{item:equivalence-product}
        $N=A^*B$ for some $A, B \in \linear(\SpaceTwo; \Space)$,
        \item
        \label{item:equivalence-alpha}
        $\Xi_\alpha \defeq \inv\alpha\Xi-\alpha^{-2}A^*A/4 \le \min\{1,\inv\alpha\}M$, and
        \item
        \label{item:equivalence-projection}
        $(M, M-\Xi_\alpha) \in \mathcal{P}(\inv T(\realoptw), \realoptu)$ when $\alpha \in (0, 1)$.
    \end{enumerate}
    Then $T$ is $(\Xi_\alpha,B^*B,M)$-partially subregular at $(\realoptu, \realoptw) \in \graph T$ if it is $(\Xi,N,M)$-partially strongly submonotone at this point.

    If \ref{item:equivalence-projection} does not hold, we still have $(\alpha\Xi_\alpha, \alpha B^*B, M)$-partial subregularity.
\end{lemma}

\begin{proof}
    By partial strong submonotonicity, for some neighbourhood $\mathcal{U}$ of $\realoptu$, for all $u^* \in \inv T(\realoptw)$, $u \in \mathcal{U}$, and $w \in T(u)$ it holds
    \begin{equation*}
        \iprod{w-\realoptw}{u-u^*}_N
        + \norm{u-u^*}_{M-\Xi}^2
        \ge \dist^2_M(u, \inv T(\realoptw)).
    \end{equation*}
    By Young's inequality and \ref{item:equivalence-product}, for any $\alpha>0$  therefore
    \begin{equation*}
        \norm{w-\realoptw}_{\alpha B^*B}^2
        + \norm{u-u^*}_{M-\Xi+A^*A/(4\alpha)}^2
        \ge
        \dist^2_M(u, \inv T(\realoptw)).
    \end{equation*}
    Since $\Xi-A^*A/(4\alpha)=\alpha\Xi_\alpha$, and $M \ge \alpha\Xi_\alpha$ by \ref{item:equivalence-alpha}, we obtain $(\alpha\Xi_\alpha, \alpha B^*B, M)$-partial subregularity after taking the infimum over $u^* \in \inv T(\realoptw)$ and $w \in T(u)$. By \ref{item:equivalence-alpha} also $M \ge \Xi_\alpha$. Minding \ref{item:equivalence-projection}, we may therefore apply \cref{lemma:subregularity-conversion}, which yields the claimed $(\Xi­_\alpha, B^*B, M)$-partial subregularity.
\end{proof}

For the converse relationship, we essentially have to assume that $T$ is a convex subdifferential.

\begin{definition}
    \label{def:gap-function}
    Let $T: \Space \setto \Space$, $N,\Gamma \in \linear(\Space; \Space)$, and $(\realoptu,\realoptw) \in \graph T$. We say that $\tilde\gap: \Space \times \Space \to \extR$ is an \term{$(N,\Gamma)$-gap function} for $T$ at $\realoptw$ if for all $u^* \in \inv T(\realoptw)$ we have $\iprod{T(u)}{u-u^*}_N \ge \tilde\gap(u; u^*) + \norm{u-u^*}_{\Gamma}^2$ with $\subdiff \tilde\gap(\freevar; u^*)=NT$, and $\tilde \gap(u^*; u^*)=0$.
\end{definition}

In the next result, we stress that $\mathcal{U}$ is open; otherwise in the present work, we do not strictly require this, and $\mathcal{U}$ could indeed be any subset of $\Space$.

\begin{lemma}
    \label{lemma:subregularity-to-submonotonicity}
    Let $T: \Space \setto \Space$ have closed values, and $\Xi, N, M, P \in \linear(\Space; \Space)$ with $M \ge 0$, and $M \ge P$.
    Suppose the following structural conditions hold:
    \begin{enumerate}[label=(\roman*)]
        \item\label{item:equivalence-regularity-to-monotonicity-1}
        $N=A^*B$ for some $A, B \in \linear(\Space; \Space)$.

        \item\label{item:equivalence-regularity-to-monotonicity-m}
        $A$ is invertible with $M\le m A^*A$ for some $m > 0$.

        \item\label{item:equivalence-regularity-to-monotonicity-2}
        $T$ admits for some $\Gamma \in \linear(\Space; \Space)$ an $(N,\Gamma)$-gap function $\tilde\gap$ at $\realoptw$.

        \item\label{item:equivalence-regularity-to-monotonicity-3}
        $M-\Xi+\Gamma$ is self-adjoint and positive semi-definite.

        \item\label{item:equivalence-regularity-to-monotonicity-4}
        $M \le \alpha P$ for some $0 < \alpha < 1/(16m)$.

        \item\label{item:equivalence-regularity-to-monotonicity-projection}
        $(M, M-\Xi_\alpha) \in \mathcal{P}(\inv T(\realoptw), \realoptu)$ when $\alpha > 1$.
    \end{enumerate}
    Then $T$ is $(\Xi,N,M)$-partially strongly submonotone at $(\realoptu, \realoptw) \in \graph T$ provided it is $(P,B^*B,M)$-partially subregular at this point. In both properties, we require the neighbourhood $\mathcal{U}$ to be open.

    We may replace \ref{item:equivalence-regularity-to-monotonicity-projection} and $(P,B^*B,M)$-partial subregularity by $(\alpha P, \alpha B^*B,M)$-partial subregularity.
\end{lemma}

\begin{remark}
	The condition \cref{item:equivalence-regularity-to-monotonicity-4} typically forces $P > 0$, so we require ``full'' subregularity in the sense of non-singularity of $P$. Together with the condition $M \ge P$, it may happen that only $M=P$ is possible.
\end{remark}

The proof follows ideas from \cite{aragon2008characterization,artacho2013metric}, adding the necessary extra work for partial regularity.

\begin{proof}
    Suppose, to reach a contradiction, that the claimed partial strong submonotonicity does not hold in any neighbourhood of $\realoptu$.
    Then for any $r>0$, we can find some $u^* \in \inv T(\realoptw)$ and
   	\begin{equation}
   		\label{eq:subregularity-to-submonotonicity-uprime-ball}
    	u' \in \tilde\B_M(\realoptu, r) \defeq \{ u \in \Space \mid \norm{\realoptu-u}_M < r,\, \norm{\realoptu-u} < r\}
    \end{equation}
    with
    \[
        \inf_{w \in T(u')} \left( \iprod{w-\realoptw}{u'-u^*}_N
        + \norm{u'-u^*}_{M-\Xi}^2 \right)
        < \dist_M^2(u', \inv T(\realoptw)).
    \]
    This implies
    \begin{equation}
        \label{eq:epsilon-def}
        f(u') < \epsilon \defeq
        \dist_M^2(u', \inv T(\realoptw)) - \norm{u'-u^*}_{M-\Xi+\Gamma}^2
    \end{equation}
    for
    \[
        f(u)
        \defeq
            \tilde\gap(u; u^*)
            -\iprod{\realoptw}{u-u^*}_N.
    \]
    Since $N \realoptw \in \subdiff \tilde\gap(u^*; u^*)$, clearly $u^*$ is a minimiser of $f$. Moreover $f(u^*)=0$ because $\tilde\gap(u^*; u^*)=0$. Thus $f \ge 0$ and $f(u') \le f(\realoptu) + \epsilon$. By Ekeland's variational principle \cite{ekeland1974variational}, given $\lambda>0$ (to be specified later), there now exits $v_\lambda \in \Space$ with
    \begin{subequations}
    \begin{gather}
        \label{eq:ekeland-dist}
        \norm{u'-v_\lambda}_{M} \le \lambda,
        \\
        f(v_\lambda) 
        \le f(u'),
        \quad\text{and}\\
        f(x) > f(v_\lambda) - (\epsilon/\lambda)\norm{x-v_\lambda}_{M}
        \quad (x \ne v_\lambda).
    \end{gather}
    \end{subequations}
    It follows
    \[
        f(x) + (\epsilon/\lambda)\norm{x-v_\lambda}_{M} > f(v_\lambda) + (\epsilon/\lambda)\norm{v_\lambda-v_\lambda}_{M}
        \quad (x \ne v_\lambda).
    \]
    Thus $v_\lambda$ is a minimiser of the convex function $x \mapsto f(x) + (\epsilon/\lambda)\norm{x-v_\lambda}_{M}$.
    The necessary first-order optimality conditions therefore state
    \[
        N(T(v_\lambda)-\realoptw) + (\epsilon/\lambda) M q \ni 0 \quad\text{for some}\quad \norm{q} \le 1.
    \]
    Since $N=A^*B$, using \ref{item:equivalence-regularity-to-monotonicity-m} we obtain
    \begin{equation}
        \label{eq:equiv-regularity-to-monotonicity-estim1}
        \dist^2_{B^*B}(\realoptw, T(v_\lambda)) \le m(\epsilon/\lambda)^2.
    \end{equation}

    Under the assumption that $T$ is $(P,B^*B, M)$-partially subregular at $(\realoptu, \realoptw) \in \graph T$, we obtain by  \ref{item:equivalence-regularity-to-monotonicity-projection} and \cref{lemma:subregularity-conversion} corresponding $(\alpha P,\alpha B^*B, M)$-partial subregularity.
    This implies
    \begin{equation}
        \label{eq:equivalence-regularity-to-monotonicity-subreg-vlambda}
        \dist^2_{M}(v_\lambda, \inv T(\realoptw))
        -
        \norm{v_\lambda-u^*}_{M- \alpha P}^2
        \le
        \dist_{\alpha B^*B}^2(\realoptw, T(v_\lambda)),
    \end{equation}
    provided $v_\lambda \in \mathcal{U}$ for $\mathcal{U}$ the neighbourhood of partial subregularity. We may without loss of generality assume that $\mathcal{U}=\tilde \B_M(\realoptu, R)$ for some $R>0$. 

    Recalling the definition of $\epsilon$ in \eqref{eq:epsilon-def}, using \ref{item:equivalence-regularity-to-monotonicity-3} and Young's inequality, we estimate
    \begin{equation}
    	\label{eq:equiv-regularity-to-monotonicity-epsilon2.0}
        \begin{split}
        \epsilon & = \dist_M^2(u', \inv T(\realoptw)) -\norm{u'-u^*}_{M-\Xi+\Gamma}^2
        \le
        \dist_M^2(u', \inv T(\realoptw)) 
        \\
        &
        \le 2[\dist_M^2(v_\lambda, \inv T(\realoptw)) - \norm{v_\lambda-u^*}_{M-\alpha P}^2]
        + \norm{v_\lambda-u^*}_{2(M-\alpha P)}^2
        + \norm{u'-v_\lambda}_{2M}^2.
        \end{split}
    \end{equation}
    Further using \cref{eq:equivalence-regularity-to-monotonicity-subreg-vlambda,eq:equiv-regularity-to-monotonicity-estim1,eq:ekeland-dist}, and \cref{item:equivalence-regularity-to-monotonicity-4}, we obtain
    \begin{equation}
        \label{eq:equiv-regularity-to-monotonicity-epsilon2}
        \epsilon 
        \le 2\alpha\dist_{B^*B}^2(\realoptw, T(v_\lambda)) + 2\lambda^2
        \le 2\alpha m(\epsilon/\lambda)^2 + 2\lambda^2.
    \end{equation}
    By \eqref{eq:epsilon-def} and $f \ge 0$, necessarily $\epsilon>0$.
    Choosing $\lambda \defeq \sqrt{m\alpha\inv\theta\epsilon}$ for some $\theta>0$, \eqref{eq:equiv-regularity-to-monotonicity-epsilon2} therefore shows
    \begin{equation}
    	\label{eq:contradiction}
        1 \le 2\theta + 2\alpha m \inv\theta.
    \end{equation}

    We recall that for \eqref{eq:equivalence-regularity-to-monotonicity-subreg-vlambda} to hold we still need to ensure  $v_\lambda \in \mathcal{U}=\tilde \B(\realoptu, R)$. By \cref{eq:subregularity-to-submonotonicity-uprime-ball,eq:ekeland-dist}, we have
    \[
    	\norm{v_\lambda-\realoptu}_M
    	\le
    	\norm{v_\lambda-u'}_M
    	+
    	\norm{u'-\realoptu}_M
    	\le
    	\lambda + r.
    \]
    Since $r>0$ was arbitrary, it suffices to ensure $\lambda < R$.    
    This will hold and we will contradict \eqref{eq:contradiction} if we choose $\theta$ to satisfy
    \begin{equation}
    	\label{eq:contradiction-theta}
         1 - \sqrt{1 - 16\alpha m} < 4\theta < 1 + \sqrt{1 - 16\alpha m},
         \quad\text{and}\quad
         m \alpha\epsilon < \theta R^2.
    \end{equation}
    The expression under the square root is positive by \cref{item:equivalence-regularity-to-monotonicity-4}. Therefore the first part of of \cref{eq:contradiction-theta} holds for some $\theta>0$.
    From the first line of \cref{eq:equiv-regularity-to-monotonicity-epsilon2.0}, we have $\epsilon \le \norm{u'-\realoptu}_M^2 \le r^2$. Therefore, the second part of \eqref{eq:contradiction-theta} can also be made to hold by taking $r>0$ small enough. This gives the desired contradiction.
\end{proof}

\subsection{Concrete results in special cases}

We now specialise the implications above to some simple cases.

Take $M=I$ and let $\Xi$ be a projection to a subspace. If $T$ is  $(\Xi, \kappa I, I)$-partially strongly submonotone, the next proposition yields $(2\Xi-I, \kappa^2 I, I)$-partial subregularity, which reads
\begin{equation*}
    \kappa^2 \dist_I^2(\realopt{w}, T(u))
    + 2\dist_{I-\Xi}^2(u, \inv T(\realopt{w}))
    \ge
    \dist_I^2(u, \inv T(\realopt{w}))
    \quad (u \in \mathcal{U}).
\end{equation*}
If $\realoptu \in \inv T(\realoptw)$ were unique, this would be the same as
\[
    \kappa^2 \dist_I^2(\realopt{w}, T(u))
    +
    \dist_{I-\Xi}^2(u, \inv T(\realopt{w}))
    \ge
    \dist_{\Xi}^2(u, \inv T(\realopt{w})).
\]
This is a stronger property than $(\Xi, \kappa^2 I, I)$-partially subregularity, for which \eqref{eq:partial-subreg} in the case of unique $\realoptu$ would read
\[
    \kappa^2 \dist_I^2(\realopt{w}, T(u))
    \ge
    \dist_{\Xi}^2(u, \inv T(\realopt{w})).
\]
Of course, if $\Xi=I$, the two properties are the same.

\begin{proposition}
    \label{prop:equivance2-subspace}
    Let $T: \Space \setto \Space$ have closed values, and $M, \Xi \in \linear(\Space; \Space)$ with $M$ positive definite and self-adjoint. Let $(\realoptu, \realoptw) \in \graph T$, $\kappa>0$.

    Then $T$ is $(2\Xi-M, \kappa^2 M, M)$-partially subregular if it is $(\Xi, \kappa M, M)$-partially strongly submonotone, $\Xi \le M$, and $(M, M-\Xi) \in \mathcal{P}(\inv T(\realoptw), \realoptu)$.

    Conversely, $T$ is $(\Xi, \kappa M, M)$-partially strongly submonotone if it is $(M, \rho M, M)$-partially subregular for some $\rho>(\kappa/4)^2$ in an open neighbourhood $\mathcal{U} \ni \realoptu$, and admits an $(M,\Gamma)$-gap function at $\realoptw$ with $0 \le M-\Xi+\Gamma$ self-adjoint.
\end{proposition}

\begin{proof}
    For the `non-converse' claim, we use \cref{lemma:submonotonicity-to-subregularity} with $A=M^{1/2}$ and $B=(\kappa M)^{1/2}$.
    Taking $\inv\alpha=2$, we have $\Xi_\alpha = 2\Xi-A^*A=2\Xi-M$.
    For condition \ref{item:equivalence-alpha} we then need our assumption $\Xi \le M$.
    Since $\alpha \in (0, 1)$, for \ref{item:equivalence-projection} we need  $(M, M-\Xi_\alpha) \in \mathcal{P}(\inv T(\realoptw), \realoptu)$, which follows from us assuming $(M, M-\Xi) \in \mathcal{P}(\inv T(\realoptw), \realoptu)$.
    Since $B^*B=\kappa M$, we see from the lemma that $(\Xi, \kappa M, M)$-submonotonicity implies $(2\Xi-M,\kappa^2 M,M)$-partial subregularity.

    For the `converse' claim, we use \cref{lemma:subregularity-to-submonotonicity}, where \ref{item:equivalence-regularity-to-monotonicity-1} and \ref{item:equivalence-regularity-to-monotonicity-m} with $m=\inv c \inv\kappa $ and $A=(c\kappa M)^{1/2}$, and $B=(\inv c\kappa M)^{1/2}$ are easily verified for any $c>0$.
    By assumption  $M-\Xi+\Gamma$ is self-adjoint and positive semi-definite, so \ref{item:equivalence-regularity-to-monotonicity-3} holds.
    For \ref{item:equivalence-regularity-to-monotonicity-4} we then require $\alpha \le c \kappa/16$ and $M \le \alpha P$.
    We also need $P \le M$.
    We can take $P=M$ and $\alpha=1$ if we take $c > 16 \inv\kappa$.
    Then $B^*B=\rho M$ for $\rho=\kappa\inv c > (\kappa/4)^2$.
    With these choices $\alpha \le 1$, so condition \ref{item:equivalence-regularity-to-monotonicity-2} also holds.
    Now \cref{lemma:subregularity-to-submonotonicity} shows that $(M, \rho M, M)$ subregularity implies $(\Xi, \kappa M, M)$-submonotonicity.
\end{proof}

The following generalises \cite[Theorem 3.3]{aragon2008characterization} from the setting of convex functions. Because we formulate subregularity using squared norms, the conversion also squares $\kappa$.

\begin{corollary}
    \label{cor:equivance2}
    Let $T: \Space \setto \Space$ have closed values, and $M \in \linear(\Space; \Space)$ be positive definite and self-adjoint. Let $(\realoptu, \realoptw) \in \graph T$, and $\kappa>0$.

    Then $T$ is $(\kappa^2 M, M)$-subregular if it is $(\kappa M, M)$-strongly submonotone.

    Conversely, if $T$ also admits an $(M,0)$-gap function at $\realoptw$, then $T$ is $(\kappa M, M)$-strongly submonotone if it is $(\rho M, M)$-subregular for some $\rho>(\kappa/4)^2$ in an open neighbourhood $\mathcal{U} \ni \realoptu$.
\end{corollary}

\begin{proof}
    We apply \cref{prop:equivance2-subspace} with $\Xi=M$ and $\Gamma=0$.
    Clearly $M \ge \Xi$, $(M, M-\Xi) \in \mathcal{P}(\inv T(\realoptw), \realoptu)$, and $M-\Xi+\Gamma$ is self-adjoint and positive semi-definite.
\end{proof}

\begin{corollary}
    \label{cor:equivance2-subdiff}
    Let $f: \Space \to \extR$ be convex, proper, and lower semicontinuous, and $M \in \linear(\Space; \Space)$ be positive definite and self-adjoint. Let $(\realoptu, \realoptw) \in \graph \subdiff f$, and $\kappa>0$.

    Then $\subdiff f$ is $(\kappa^2 I, I)$-subregular if it is $(\kappa I, I)$-strongly submonotone.

    Conversely, $\subdiff f$ is $(\kappa I, I)$-strongly submonotone if it is $(\rho I, I)$-subregular for some $\rho>(\kappa/4)^2$ in an open neighbourhood $\mathcal{U} \ni \realoptu$.
\end{corollary}

\begin{proof}
    Since $f$ is convex, $\subdiff f$ admits the $(I, 0)$-gap function $\tilde \gap(u; u^*) \defeq f(u)-f(u^*)$.
    Now we use \cref{cor:equivance2} with $M=I$.
\end{proof}

\begin{remark}
    Each of the results in this section easily extends to submonotonicity of the form $(\Xi, \kappa N, M)$ for $N=M^{1/2}B$, and some $\kappa>0$, $B, \Xi \in \linear(\Space;\Space)$ if the corresponding subregularity is of the form $(P, \kappa' B^*B, M)$ for some $P$ and $\kappa'>0$.
\end{remark}

\subsection{Examples}
\label{sec:examples-relationships}

The next example demonstrates that the `non-converse' direction of \cref{prop:equivance2-subspace} cannot in the general case be improved to yield subregularity of similar ``partiality'' $\Xi$ as the underlying submonotonicity.

\begin{example}[Partial submonotonicity without corresponding subregularity]
	On $\R^2$, for $u=(u_1, u_2)$ and some $\gamma \in [0, 1]$, let us define
	\[
		T(u) \defeq
		\begin{cases}
			\left(\gamma u_1, \sqrt{(1-\gamma^2)u_1^2-u_2^2}\right), & (1-\gamma^2)u_1^2 \ge u_2^2 \text{ and } u_2 \ge 0, \\
			\emptyset, & \text{otherwise}.
		\end{cases}
	\]
	Let $\Xi \defeq\left(\begin{smallmatrix} \gamma & 0 \\ 0 & 0 \end{smallmatrix}\right)$.	Then $T$ is $(\Xi, I, I)$-partially strongly monotone at $(\realoptu, \realoptw) \defeq (0, 0) \in \graph T$. However, $T$ is not $(\Xi, I, I)$-partially subregular.
\end{example}

\begin{demonstration}	
	Clearly $\inv T(0)=0$.
	With $N \defeq M \defeq I$, \eqref{eq:submonotonicity} expands as
	\[
		\left[\gamma u_1^2 + u_2 \sqrt{(1-\gamma^2)u_1^2-u_2^2}\right] + (1-\gamma)u_1^2 + u_2^2 \ge u_1^2+u_2^2
		\quad ((1-\gamma^2)u_1^2 \ge u_2^2,\, u_2 \ge 0).
	\]
	This clearly holds.

    Since $\inv T(0)=\{0\}$ is a singleton, the projection condition of \cref{prop:equivance2-subspace} clearly holds. We also have $I \ge \Xi$.
    Referring to the proposition we therefore obtain $(2\Xi-I, I, I)$-partial subregularity, for which \eqref{eq:partial-subreg} expands as
	\[
		[\gamma^2u_1^2 + ((1-\gamma^2)u_1^2-u_2^2)] + 2[(1-\gamma)u_1^2 + u_2^2] \ge u_1^2 + u_2^2
		\quad ((1-\gamma^2)u_1^2 \ge u_2^2,\, u_2 \ge 0).
	\]
	This clearly holds.
	On the other hand \eqref{eq:partial-subreg} for $(\Xi, I, I)$-partial subregularity would require
	\[
		[\gamma^2u_1^2 + ((1-\gamma^2)u_1^2-u_2^2)] + [(1-\gamma)u_1^2 + u_2^2] \ge u_1^2 + u_2^2
		\quad ((1-\gamma^2)u_1^2 \ge u_2^2,\, u_2 \ge 0).
	\]
	This can be simplified as $(1-\gamma)u_1^2 \ge u_2^2$. Taking $u_2^2=(1-\gamma^2)u_1^2$, we see that the condition cannot hold for $\gamma \in (0, 1)$.
\end{demonstration}

Finally, having already showed the converse, we show that subregularity is also not a weaker property than strong submonotonicity.

\begin{example}[Subregularity without \emph{strong} submonotonicity]
	\label{ex:no-strong-mono}
	On $\R^2$, with $u=(u_1, u_2)$, let
	\[
		T(u)=
		\begin{cases}
			(u_2, u_1), & u_1, u_2 \ge 0, \\
			\emptyset, & \text{otherwise}.
		\end{cases}
	\]
    Then $T$ is $(I, I, I)$-partially subregular at $(\realoptu,\realoptw)=0 \in \graph T$, but $(\Xi, I, I)$-partially strongly submonotone with $0 \le \Xi \le I$ only for $\Xi=0$.
\end{example}

\begin{demonstration}
    Clearly $0=\realoptu \in \inv T(0)$ is unique. Therefore, since $\iprod{T(u)}{u}=2 u_1u_2$ for $u_1, u_2 \ge 0$, \eqref{eq:submonotonicity} for $M=N=I$ and arbitrary $0 \le \Xi \le I$ reads
	\[
        2u_1u_2 
        \ge \norm{(u_1,u_2)}^2_{\Xi} \quad (u_1,u_2 \ge 0).
    \]
    This clearly holds for $\Xi=0$.
    Generally, for 	$
    \Xi \le \begin{psmallmatrix}
        a & b \\
        c & d
    \end{psmallmatrix}
    $,
    it is necessary that $2 \ge b+c$ and $a, d=0$.
    The only positive semi-definite choice that satisfies this is $\Xi=0$.

    On the other hand, \eqref{eq:partial-subreg} for $(I, I, I)$-partially subregularity  with $\mathcal{U}=\Space$ reads
    \[
        \norm{T(u_1, u_2)}^2 \ge \norm{(u_1,u_2)}^2 \quad (u_1,u_2 \ge 0).
    \]
    Since $\norm{T(u)}^2 = u_2^2+u_1^2$, this immediately holds.
\end{demonstration}

\section{Primal-dual methods}
\label{sec:saddle}

We now apply the work of the previous sections to primal-dual methods.
We concentrate on submonotonicity, as it appears in general to be easier to prove directly than subregularity. If in a specific application subregularity is known, it can using the results of \cref{sec:relationships} be converted into submonotonicity.
We start with general results on saddle-point problems, follow with the primal-dual proximal splitting method (PDPS) of \cite{chambolle2010first}, and finish with examples regarding specific applications.

\subsection{Saddle-point problems}

Many problems in data science and image processing can be written in the form
\begin{equation}
    \label{eq:prob-primal}
    \min_x G(x) + F(Kx),
\end{equation}
for $G: X \to \extR$ and $F: Y \to \extR$ be convex, proper and lower semicontinuous, and  $K \in \linear(X; Y)$. In image processing one would often work in the Banach space of functions of bounded variation, but after discretisation, if necessary, we can assume that the spaces $X$ and $Y$ are Hilbert (and finite-dimensional).

It can be difficult to apply an optimisation algorithm directly to \eqref{eq:prob-primal}: $F$ is typically nonsmooth, so gradient steps are out of the question. Due to the coupling effects of $K$, also a proximal step for $F \circ K$ is not feasible. This is why we are interested in the equivalent saddle point problem. For $F^*$ the convex conjugate of $F$, this can be written
\begin{equation}
    \label{eq:saddle}
    \min_x ~\max_y~ G(x) + \iprod{Kx}{y} - F^*(y).
\end{equation}

The first-order necessary optimality conditions for \eqref{eq:saddle} can be written
\begin{equation}
    \label{eq:oc}
    -K^* \realopty \in \subdiff G(\realoptx),
    \quad\text{and}\quad
    K \realoptx \in \subdiff F^*(\realopty).
\end{equation}
Setting $\Space \defeq X \times Y$ and introducing the variable splitting notation $u=(x, y)$, $\realoptu=(\realoptx,\realopty)$, etc., this can succinctly be written as $0 \in H(\realoptu)$ in terms of the operator
\begin{equation}
    \label{eq:h}
    H(u) \defeq
        \begin{pmatrix}
            \subdiff G(x) + K^* y \\
            \subdiff F^*(y) -K x
        \end{pmatrix}.
\end{equation}
From now on, we will not explicitly assume $G$ and $F^*$ to be convex functions, indeed $\subdiff G$ and $\subdiff F^*$ can be replaced by any set-valued operators. In particular, they can be nonconvex functions, and $\subdiff$ can be replaced by the Clarke, limiting, or other subdifferential of nonconvex functions. In this case, it should be shown that \eqref{eq:oc} characterises critical points to \eqref{eq:saddle} or \eqref{eq:prob-primal}.

For some primal and dual step length $\tau_i, \sigma_{i+1}>0$ and testing parameters $\sigmaTest_i,\sigmaTest_{k+1}>0$, we take
\begin{equation}
    \label{eq:test}
    \Test_{i+1} \defeq
        \begin{pmatrix}
            \tauTest_i\tau_i I & 0 \\
            0 & \sigmaTest_{i+1}\sigma_{i+1} I
        \end{pmatrix}.
\end{equation}
For some $\gamma,\rho \ge 0$ we also introduce
\begin{equation}
    \label{eq:saddle-xi}
    \Xi_{i+1}^0 \defeq \begin{pmatrix}
        0  & 2\tau_i K^* \\
        -2\sigma_{i+1}K & 0
    \end{pmatrix},
    \quad\text{and}\quad
    \Xi_{i+1} \defeq \Xi_{i+1}^0 + \begin{pmatrix} \tau_i \gamma & 0 \\ 0 & \sigma_{i+1} \rho \end{pmatrix}.
\end{equation}

We want to use \cref{thm:convergence-result-submonotone} to prove the convergence of \eqref{eq:pp} for $H$ defined in \eqref{eq:h} for some linear preconditioner $\Precond_{k+1} \in \linear(U; U)$.
To do so, we need to prove the $(\Test_{i+1}\Xi_{i+1}, 2\Test_{i+1}, \Test_{i+2}\Precond_{i+2})$-partial strong submonotonicity for $H$ at $(\realoptu, 0) \in \graph H$. Dividing \eqref{eq:submonotonicity} by $2$, this amounts to showing
\begin{equation}
    \label{eq:h-submonotonicity}
    \inf_{u^* \in \inv H(0)}
    \left(
    \iprod{w}{u-u^*}_{\Test_{i+1}}
    + \frac{1}{2}\norm{u-u^*}_{\Test_{i+2}\Precond_{i+2}-\Test_{i+1}\Xi_{i+1}}^2
    \right)
    \ge \frac{1}{2}\dist^2_{\Test_{i+2}\Precond_{i+2}}(u, \inv H(0))
\end{equation}
for all $u \in \mathcal{U}$ and $w \in H(u)$ in some neighbourhood $\mathcal{U}$ of $\realoptu$.

By \cref{eq:h}, for some $q \in \subdiff G(u)$, $z \in \subdiff F^*(y)$, $q^* \defeq -K^*y^* \in \subdiff G(x^*)$, and $z^* \defeq Kx^* \in \subdiff F^*(y^*)$ we have
\begin{equation*}
    \begin{split}
    \iprod{w}{u-u^*}_{\Test_{i+1}}
    &
    =
    \tauTest_i\tau_i \iprod{q-q^*}{x-x^*} + \sigmaTest_{i+1}\sigma_{i+1} \iprod{z-z^*}{y-y^*}
    + \frac{1}{2}\norm{u-u^*}_{\Test_{i+1}\Xi_{i+1}^0}^2
    \end{split}
\end{equation*}
Recalling \eqref{eq:saddle-xi}, to show \eqref{eq:h-submonotonicity}, it remains to prove
\begin{multline*}
    \inf_{u^* \in \inv H(0)}
    \Bigl(
     \tauTest_i\tau_i[\iprod{q-q^*}{x-x^*}-\frac{\gamma}{2} \norm{x-x^*}^2]
     + \sigmaTest_{i+1}\sigma_{i+1}[\iprod{z-z^*}{y-y^*}-\frac{\rho}{2}\norm{y-y^*}^2]
     \\
      + \frac{1}{2}\norm{u-u^*}_{\Test_{i+2}\Precond_{i+2}}^2
    \Bigr)
    \ge \inf_{u^{**} \in \inv H(0)} \frac{1}{2}\norm{u-u^{**}}_{\Test_{i+2}\Precond_{i+2}}^2.
\end{multline*}
Splitting the infimum on the left hand side over the three terms, we obtain the ``marginalised'' condition:

\begin{proposition}
	\label{prop:marginalisation}
	Let  the operator $H$ defined in \eqref{eq:h}, and $\Xi_{i+1}$ in \eqref{eq:saddle-xi} for some $\gamma \ge 0$ and $\rho \ge 0$.
	For any $u^*=(x^*, y^*) \in \inv H(0)$, let $q^* \defeq -K^*y^* \in \subdiff G(x^*)$, and $z^* \defeq Kx^* \in \subdiff F^*(y^*)$. Suppose for $u=(x, y) \in \mathcal{U}$ in some neighbourhood $\mathcal{U}$ of $\realoptu$ it holds
	\begin{subequations}
	\label{eq:marginalised-submonotonicity}
	\begin{align}
		\label{eq:marginalised-submonotonicity-g}
		\inf_{u^* \in \inv H(0)} [\iprod{\subdiff G(x)-q^*}{x-x^*}-\frac{\gamma}{2} \norm{x-x^*}^2] & \ge 0, 
		\quad\text{and} \\
		\label{eq:marginalised-submonotonicity-fstar}
		\inf_{u^* \in \inv H(0)} [\iprod{\subdiff F^*(y)-z^*}{y-y^*}-\frac{\rho}{2}\norm{y-y^*}^2] & \ge 0.
	\end{align}
	\end{subequations}
	Then $H$ is $(\Test_{i+1}\Xi_{i+1}, 2\Test_{i+1}, \Test_{i+2}\Precond_{i+2})$-partially strongly submonotone at $(\realoptu, 0) \in \graph H$.
\end{proposition}

The conditions \eqref{eq:marginalised-submonotonicity} amount to $(\gamma I, 2I, 0)$-partial strong submonotonicity, which is simply local strong monotonicity at every respective $x^*$ for $q^*$ or $y^*$ for $z^*$ with $(x^*, y^*) \in \inv H(0)$. It is a much stronger requirement than “proper” submonotonicity, but weaker than strong monotonicity.
A challenging alternative is to try to prove the $(\Test_{i+1}\Xi_{i+1}, 2\Test_{i+1}, \Test_{i+2}\Precond_{i+2})$-partial strong submonotonicity of $H$ directly.

\subsection{The primal-dual proximal splitting}

The primal-dual proximal splitting (PDPS) of \cite{chambolle2010first} consists of iterating the system
\begin{align*}
    \nextx & \defeq (I+\tau_i \subdiff G)^{-1}(\thisx - \tau_i K^* y^{i}),\\
    \overnextx & \defeq \omega_i (\nextx-\thisx)+\nextx, \\
    \nexty & \defeq (I+\sigma_{i+1} \subdiff F^*)^{-1}(\thisy + \sigma_{i+1} K \overnextx).
\end{align*} 
Setting $\omega_i \defeq \inv\sigmaTest_{i+1}\inv\sigma_{i+1}\tauTest_i\tau_i$,
according to \cite{he2012convergence}, see also \cite{tuomov-cpaccel,tuomov-proxtest}, this can also be obtained from \eqref{eq:pp} with 
\begin{equation}
	\label{eq:pdhgm-precond}
    \Precond_{i+1}=\begin{pmatrix} \inv\tau_i I & - K^* \\ -\omega_i K & \inv\sigma_{i+1} I\end{pmatrix}
    \quad\text{so that}\quad
    \Test_{i+1}\Precond_{i+1}=\begin{pmatrix} \tauTest_i I & -\tauTest_i\tau_i K^* \\ -\tauTest_i\tau_i K & \sigmaTest_{i+1} I\end{pmatrix}.
\end{equation}
The role of the top-right off-diagonal term in $\Precond_{i+1}$ is to decouple the update of the primal variable $\nextx$ to not depend on $\nexty$, and so make the algorithm computable. The bottom-left off-diagonal is then chosen to ensure the self-adjointness of $\Test_{i+1}\Precond_{i+1}$.\footnote{The PDPS is sometimes called the Primal-Dual Hybrid Gradient Method, Modified (PDHGM) \cite{esser2010general}. The word “modified” refers to this self-adjoint part. The ``unmodified'' PDHG would have zero in this corner \cite{zhu2008efficient,he2012convergence}.}

\begin{theorem}
    \label{thm:pdps-convergence}
	Take $H$ as in \eqref{eq:h}, $\Test_{i+1}$ as in \eqref{eq:test}, and $\Precond_{i+1}$ as in \eqref{eq:pdhgm-precond}. Also let $\Xi_{i+1}$ be as in \eqref{eq:saddle-xi} for some $\gamma,\rho \ge 0$. Suppose the testing and step length parameters $\tau_i,\sigma_i,\tauTest_i$, and $\sigmaTest_i$  satisfy for some $\delta \in (0, 1)$,
	\begin{subequations}%
	\label{eq:pdhgm-step-cond}%
	\begin{align}
		\label{eq:pdhgm-step-cond-1}
		\tauTest_{i+1} & \le \tauTest_i(1+\gamma\tau_i),
		&
		\sigmaTest_{i+2} & \le \sigmaTest_{i+1}(1+\rho\sigma_{i+1}),
		\quad\text{and}
		\\
		\label{eq:pdhgm-step-cond-2}
		\tauTest_i\tau_i &=\sigmaTest_i\sigma_i,
		&
		\inv\sigmaTest_i\sigmaTest_{i+1} & \ge (1-\delta)\tau_i\sigma_i \norm{K}^2.
	\end{align}%
	\end{subequations}%
	Suppose $H$ is $(\Test_{i+1}\Xi_{i+1}, 2\Test_{i+1}, \Test_{i+2}\Precond_{i+2})$-partially strongly submonotone at $(\realoptu, 0) \in \graph H$ with the neighbourhood of submonotonicity satisfying $\{\thisu\}_{i=0}^N \subset \mathcal{U}$.
	Then the condition \eqref{eq:convergence-fundamental-condition-iter-h-submonotonicity} of \cref{thm:convergence-result-submonotone} holds, and
	\begin{equation}
		\label{eq:pdhgm-test-precond-lower-approx}
		\Test_{i+1}\Precond_{i+1} \ge \begin{pmatrix}\delta \tauTest_i I & 0 \\ 0 & 0 \end{pmatrix}.
	\end{equation}
	In particular, we have the convergence rate estimate
	\begin{equation}
		\label{eq:pdhgm-rate}
		\dist^2(x^N; \hat X) \le \inv \tauTest_N \inv\delta \dist^2_{\Test_1\Precond_1}(u^0; \inv H(0))
		\quad\text{where}\quad
		 \hat X \defeq \{ x^* \mid (x^*, y^*) \in \inv H(0)\}.
	\end{equation}
\end{theorem}

\begin{proof}
	It is easy to see that $\Test_{i+1}(\Precond_{i+1}+\Xi_{i+1}) \ge \Test_{i+2}\Precond_{i+2}$ if we enforce the first three conditions of \eqref{eq:pdhgm-step-cond} \cite{tuomov-proxtest}. To ensure \eqref{eq:pdhgm-test-precond-lower-approx}, using Young's inequality in \cref{eq:pdhgm-precond}, we derive the condition $\sigmaTest_{i+1} \ge (1-\delta)\tauTest_i\tau_i^2 \norm{K}^2$.
	Using the third condition of \eqref{eq:pdhgm-step-cond}, this can alternatively be written as the last condition of \eqref{eq:pdhgm-step-cond}.
	Finally, \cref{thm:convergence-result-submonotone} shows the basic estimate \eqref{eq:convergence-result-main-h}. Using \eqref{eq:pdhgm-test-precond-lower-approx} in \eqref{eq:convergence-result-main-h}, we derive \eqref{eq:pdhgm-rate}.
\end{proof}

We now need to satisfy \eqref{eq:pdhgm-step-cond} and, minding \eqref{eq:pdhgm-rate}, to derive convergence rates, try to make $\tauTest_i$ grow as fast as possible.
The latter analysis is also based on the rules \eqref{eq:pdhgm-step-cond}, and proceeds exactly as in \cite{tuomov-proxtest,clasonvalkonen2020nonsmooth} (where $\rho$ and $\gamma$ arise from mere (strong) monotonicity).

\begin{example}[Weak convergence]
	\label{ex:pdhgm-weak}
	If $\rho=0$ and $\gamma=0$, by \eqref{eq:pdhgm-step-cond} $\tauTest_i$ and $\sigmaTest_{i+1}$ remain constant, so we obtain no convergence rates, merely weak convergence of subsequences of iterates; see \cite{tuomov-proxtest,clasonvalkonen2020nonsmooth}.
	By \eqref{eq:pdhgm-step-cond} the step lengths $\tau_i=\tau$ and $\sigma_i=\sigma$ need to be constant and satisfy $1 \ge (1-\delta)\tau\sigma \norm{K}^2$.
	It is also possible \cite{chambolle2010first,tuomov-proxtest} to derive $O(1/N)$ convergence of an ergodic duality gap, which we have avoided introducing here. Note that $\omega_i \equiv 1$.
\end{example}

\begin{example}[Acceleration to $O(1/N^2)$]
	\label{ex:pdhgm-accel}
	If $\rho=0$ and $\gamma > 0$, we still have $\sigmaTest_{i+1}=\sigmaTest_0$ constant. We can initialise $1 \ge (1-\delta)\tau_0\sigma_0 \norm{K}^2$ and update $\sigma_{i+1} \defeq \sigma_i\tau_i/\tau_{i+1}$. Taking
	$\tau_i=\tauTest_i^{-1/2}$, and $\tauTest_{i+1} \defeq \tauTest_i(1+\tilde\gamma\tau_i)$ for some $\tilde\gamma \in (0, \gamma]$, we can prove that $\tauTest_i$ is of the order $\Theta(N^2)$. Therefore we obtain from \eqref{eq:pdhgm-rate} $O(1/N^2)$ convergence rates for the primal variable $x$.
    Note that $\omega_i = \inv\sigmaTest_{i+1}\inv\sigma_{i+1}\tauTest_i\tau_i = \inv\tauTest_{i+1}\inv\tau_{i+1}\tauTest_i\tau_i=\tau_{i+1}/\tau_i<1$.
\end{example}

\begin{example}[Linear rates]
	\label{ex:pdhgm-linear}
    If both $\rho > 0$ and $\gamma > 0$, take  $\sigma_{i+1}=\sigma$ and $\tau_i=\tau$ as constants satisfying
    \[
        \theta \defeq 1+\min\{\rho\sigma,\gamma\tau\}  \ge (1-\delta) \tau\sigma\norm{K}^2.
    \]
    Then with $\sigmaTest_0=\tauTest_0\tau\inv\sigma$, $\tauTest_{i+1}\defeq \theta\tauTest_i$, and $\sigmaTest_{i+1}\defeq \theta\sigmaTest_i$ we verify \eqref{eq:pdhgm-step-cond}. 
    Note that $c>1$, so $\{\tauTest_i\}_{i \in \N}$ grows exponentially.
    Thus we obtain linear convergence of the primal variables from \eqref{eq:pdhgm-test-precond-lower-approx} and \eqref{eq:pdhgm-rate}.{\footnotemark} %
    Note that $\omega_i = \inv\sigmaTest_{i+1}\inv\sigma_{i+1}\tauTest_i\tau_i = 1$ as in \cref{ex:pdhgm-weak}.
\end{example}
\footnotetext{It is possible to improve \eqref{eq:pdhgm-test-precond-lower-approx} for linear convergence of the dual variables.}

\begin{remark}[Global and local convergence]
    \label{rem:global}
    Convergence results based on submonotonicity are generally local due to the neighbourhood $\mathcal{U}$. 
	With $G$ and $F^*$ convex, $H$ is monotone ($\gamma=0$ and $\rho=0$).
	Therefore the weak convergence results of \cref{ex:pdhgm-weak} are global.
    This is an \emph{a priori} convergence property. In finite dimensions, the convergence is strong, so eventually the iterates will be in the neighbourhood $\mathcal{U}$ (for some $\realoptu \in \inv H(0)$) of partial strong submonotonicity. Then the algorithm switches to faster convergence. This is only known \emph{a posteriori} as it depends on $\realoptu$.
\end{remark}

\subsection{Application examples}
\label{sec:applications}

We now look at regularity and convergence rate results for a few prototypical applications. 
We will show that linear convergence of the PDPS can in specific cases be obtained by the unaccelerated primal-dual algorithm without full strong convexity. These results improve known convergence results for this algorithm. These results are comparable to those obtained in \cite{lewis2013partial,liang2014local} for other classes of algorithms based on a related theory of smooth submanifolds.
Throughout, we rely on the marginalisation of submonotonicity in \cref{prop:marginalisation}, so seek to verify \eqref{eq:marginalised-submonotonicity}.

\begin{example}[Lasso]
	\label{ex:lasso}
	For some data matrix and vector $K \in \R^{m \times n}$ and $b \in \R^m$, and regularisation parameter $\alpha>0$, consider the Lasso or regularised regression problem
	\[
		\min_x~ \frac{1}{2}\norm{b-Kx}^2 + \alpha \norm{x}­_1.
	\]
    This can be written in the saddle point form \eqref{eq:saddle} with $G(x)=\alpha \norm{x}_1$, and $F^*(y)=\frac{1}{2}\norm{y}^2-\iprod{b}{y}$.
    Let $(\realoptx, \realopty) \in \inv H(0)$ for $H$ as in \eqref{eq:h}.
    The PDPS converges linearly whenever the primal solution $\realoptx=0$ and is unique, and any corresponding dual solution is strictly complementary ($\abs{[K^*\realopty]_k} < 1$ for all coordinates $k=1,\ldots,n$).
\end{example}

\begin{demonstration}
    We use \cref{prop:marginalisation} with \cref{thm:pdps-convergence,ex:pdhgm-linear}.
    We therefore need to verify \eqref{eq:marginalised-submonotonicity}.
	Since $F^*$ is strongly convex, we can take $\rho=1$ in \eqref{eq:marginalised-submonotonicity-fstar}.
    To prove \eqref{eq:marginalised-submonotonicity-g} for $G(x)=\sum_{k=1}^n \abs{x_k}$, we proceed as in \cref{lemma:subreg-1norm}, however \eqref{eq:marginalised-submonotonicity-g} is much more restrictive than the $(I, \gamma I)$-strong submonotonicity shown in the lemma.
    Let $(x^*, y^*) \in \inv H(0)$.
    Write $q^* \defeq (q_1^*,\ldots,q_n^*)=-K^*y^* \in \subdiff G(x^*)$ and $x^*=(x_1^*,\ldots,x_n^*)$.
    Of course $q_k^* \in \subdiff\abs{\freevar}(0)$.
    We need to prove
    \begin{equation}
        \label{eq:lasso:pdps:cond}
        \iprod{q_k-q_k^*}{x_k-x_k^*} \ge \frac{\gamma}{2}\abs{x_k-x_k^*}^2
        \quad (q_k \in \subdiff\abs{\freevar}(x_k);\, k=1,\ldots,n).
    \end{equation}
    If $x_k^* \ne 0$ for some $k$, we are only able to take $\gamma=0$ using the monotonicity of $\subdiff\abs{\freevar}$.
    However, by assumption, $x_k^* = 0$ for all $k$.
    Fix $k$ and let $\gamma>0$ be arbitrary. If $x_k>0$, $q_k=1 \ge q_k^*$. Under strictly complementarity the inequality is strict,  so \eqref{eq:lasso:pdps:cond} for this $k$ holds in a small enough neighbourhood of $x_k^*=0$.
    The case $x_k<0$ is analogous.
    This proves, via  \cref{prop:marginalisation}, the submonotonicity of $H$ required by  \cref{thm:pdps-convergence}.
    Finally, \cref{ex:pdhgm-linear} proves \cref{eq:pdhgm-step-cond} with testing parameters that yield linear convergence. By \cref{rem:global} the convergence is global.
\end{demonstration}

\begin{example}[Denoising-type problems]
	\label{ex:tv}
	Consider then for $K \in \R^{m \times n}$ and vector $b \in \R^n$ a problem of the form
	\[
		\min_x~ \frac{1}{2}\norm{b-x}^2 + \alpha \norm{Kx}_1,
	\]
    With $K$ a discretised gradient operator this includes in particular total variation denoising (anisotropic, for simplicity) and with $K$ representing decomposition into a wavelet basis, wavelet-based denoising.
    In the saddle point form the problem is
    \[
        \min_x\max_y~\frac{1}{2}\norm{b-x}^2 + \iprod{Kx}{y} + \delta_{[-\alpha, \alpha]^m}(y).
    \]
    Since the indicator function $F^*=\delta_{[-\alpha, \alpha]^m}$ is not strongly convex, but $G(x)=\frac{1}{2}\norm{b-x}^2$ is strongly convex, standard results \cite{chambolle2010first} give the $O(1/N^2)$ rate for the accelerated PDPS.

    Let $(\realoptx, \realopty) \in \inv H(0)$ for $H$ as in \eqref{eq:h}.
    Using subregularity the PDPS, in fact, converges linearly whenever $[K \realoptx]_k \ne 0$ for all $k=1,\ldots,m$. In the TV denoising case, this means that the method convergences linearly when the solution image does not contain any flat regions, while in the wavelet denoising case this means that all wavelet coefficients of the solution have to be non-zero.
\end{example}

\begin{demonstration}
    We again use \cref{prop:marginalisation} with \cref{thm:pdps-convergence,ex:pdhgm-linear}.
    We therefore need to verify \eqref{eq:marginalised-submonotonicity}.
    Since $G$ is strongly convex, we can take $\gamma=1$ in \eqref{eq:marginalised-submonotonicity-g}.
    To prove \eqref{eq:marginalised-submonotonicity-fstar}, we proceed as in \cref{lemma:subreg-ball-indicator}, however \eqref{eq:marginalised-submonotonicity-fstar} is much more restrictive than the $(I, \gamma I)$-strong submonotonicity shown in the lemma.
    Let $(x^*, y^*) \in \inv H(0)$. Observe that $x^*$ is unique by the strong convexity of $G$.
    Write $z^* = z^*=(z_1^*,\ldots,z_m^*) \defeq Kx^* \in \subdiff F^*(y^*)$.
    Then $z_k^* \in \subdiff\delta_{[-\alpha,\alpha]}(y_k^*)$.
    We need to prove 
    \begin{equation}
        \label{eq:tv:pdps:cond}
        \iprod{z_k-z_k^*}{y_k-y_k^*} \ge \frac{\rho}{2}\abs{y_k-y_k^*}^2
        \quad (z_k \in \subdiff\delta_{[-\alpha,\alpha]}(y_k);\, k=1,\ldots,m).
    \end{equation}
    If $z_k^* = 0$ for some $k$, we are only able to take $\rho=0$ using the monotonicity of $\subdiff F^*$.
    Otherwise, let $\rho>0$ be arbitrary and $k$ fixed. If $z^*_k > 0$, we have $y_k^*=\alpha$. 
    Then $y_k \le y_k^*$. It is clear that \eqref{eq:tv:pdps:cond} holds for $y_k=y_k^*$ where as for $y_k<y_k^*$, we have $z_k=0$. Thus the left hand side of \eqref{eq:tv:pdps:cond} is positive, and the inequality holds for $y_k$ in a small enough neighbourhood around $y_k^*=\alpha$.
    The case $z^*_k<0$ is similar. 
    This proves, via  \cref{prop:marginalisation}, the submonotonicity of $H$ required by  \cref{thm:pdps-convergence}.
    Finally, \cref{ex:pdhgm-linear} proves \cref{eq:pdhgm-step-cond} with testing parameters that yield linear convergence. By \cref{rem:global} the convergence is global.
\end{demonstration}

\section{Conclusions}
\label{sec:conclusions}

We have studied notions of partial strong submonotonicity and partial subregularity motivated by convergence proofs of optimisation methods.
We have showed that these concepts can be used to a show improved linear convergence rates, where conventional strong convexity or strong monotonicity is not present. 
To facilitate the verification of partial submonotonicity, in particular for saddle point problems, in further research we would like to develop a characterisation of partial submonotonicity based on nonsmooth derivatives, similar in spirit to the Mordukhovich criterion for the Aubin property.
The latter we have studied in context of the PDPS for nonlinear $K$ in \cite{tuomov-pdex2stability}.
Considering the non-trivial characterisations of metric subregularity in \cite{gfrerer2011first}, this may or may not be a fruitful path.
For local monotonicity and hypomonotonicity, which do not involve the infimisation present in submonotonicity, coderivative characterisations are considered in \cite{mordukhovich2016local,chieu2015coderivative}.

\section*{Acknowledgements}

This research has been supported by the EPSRC First Grant EP/P021298/1, ``PARTIAL Analysis of Relations in Tasks of Inversion for Algorithmic Leverage''.
The author would also like to thank Christian Clason for pointing out the papers \cite{aragon2008characterization,artacho2013metric}.

\subsection*{A data statement for the EPSRC}
This is a theory paper that does not rely on any data or program code.

 \providecommand{\eprint}[1]{\href{http://arxiv.org/abs/#1}{arXiv:#1}}
  \providecommand{\eprint}[1]{\href{http://arxiv.org/abs/#1}{arXiv:#1}}
  \providecommand{\noopsort}[1]{}


\end{document}